\documentclass[11pt]{amsart}

\usepackage[utf8]{inputenc}
\usepackage[T1]{fontenc}

\usepackage{amsmath,amsthm,esint}
\usepackage{caption}
\usepackage{mathtools}
\usepackage{cite}
\usepackage{enumitem}
\usepackage{color}
\usepackage{url}
\usepackage[usenames,dvipsnames]{xcolor}
\usepackage{csquotes}
\usepackage[UKenglish,USenglish]{babel}
\usepackage{graphicx}
\usepackage{xspace}
\usepackage{hyperref}
\usepackage{cleveref}
\usepackage{leftidx}
\usepackage[margin=1.5in]{geometry}

\usepackage[curve]{xypic}

\usepackage{comment}

\usepackage{mathrsfs}

\newcommand{\gab }{h_{a,b}}
\usepackage{hyperref}
\usepackage{color}
\definecolor{darkred}{rgb}{0.4,0,0}
\definecolor{darkgreen}{rgb}{0,0.5,0}
\definecolor{darkblue}{rgb}{0,0,0.4}
\hypersetup{
	pdftitle={PDFTitle},		
	pdfauthor={PY},	
	pdfsubject={math},		
	pdfkeywords={equidistribution, homogeneous space, Dirichlet improvability},	
	pdfnewwindow=true,		
	colorlinks=true,		
	linkcolor=darkblue,		
	citecolor=darkred,		
	filecolor=darkblue,		
	urlcolor=darkblue,		
	pdfborder={0 0 0},
	breaklinks=true
}

\newcommand{\GL}{\operatorname{GL}}
\newcommand{\SL}{\operatorname{SL}}
\newcommand{\SO}{\operatorname{SO}}

\newcommand{\Stab}{\operatorname{Stab}}

\usepackage{amssymb}

\makeatletter
\renewcommand{\paragraph}{%
	\@startsection {paragraph}{4}
	{\z@} \z@ {-\fontdimen 2\font }\bfseries
}
\makeatother

\makeatletter 
\def\@cite#1#2{{\m@th\upshape\bfseries%
		[{#1\if@tempswa{\m@th\upshape\mdseries, #2}\fi}]}}
\makeatother 

\numberwithin{equation}{section}

\theoremstyle{plain}
\newtheorem{thm}{Theorem}[section]
\newtheorem{cor}[thm]{Corollary}
\newtheorem{prop}[thm]{Proposition}
\newtheorem{lem}[thm]{Lemma}

\theoremstyle{definition}

\theoremstyle{remark}

\newcommand{\nc}{\newcommand}
\newcommand{\rnc}{\renewcommand}
\newcommand{\lip}{\left\langle}
\newcommand{\rip}{\right\rangle}
\newcommand{\ann}[1]{}
\newcommand{\eb}[1]{\emph{\textbf{#1}}}

\renewenvironment{pmatrix}{\left(\begin{smallmatrix}}{\end{smallmatrix}\right)}

\nc\bfA{\mathbf{A}}
\nc\bfB{\mathbf{B}}
\nc\bfC{\mathbf{C}}
\nc\bfD{\mathbf{D}}
\nc\bfE{\mathbf{E}}
\nc\bfF{\mathbf{F}}
\nc\bfG{\mathbf{G}}
\nc\bfH{\mathbf{H}}
\nc\bfI{\mathbf{I}}
\nc\bfJ{\mathbf{J}} 
\nc\bfK{\mathbf{K}}
\nc\bfL{\mathbf{L}}
\nc\bfM{\mathbf{M}}
\nc\bfN{\mathbf{N}}
\nc\bfO{\mathbf{O}}
\nc\bfP{\mathbf{P}}
\nc\bfQ{\mathbf{Q}}
\nc\bfR{\mathbf{R}}
\nc\bfS{\mathbf{S}}
\nc\bfT{\mathbf{T}}
\nc\bfU{\mathbf{U}}
\nc\bfV{\mathbf{V}}
\nc\bfW{\mathbf{W}}
\nc\bfY{\mathbf{Y}}
\nc\bfX{\mathbf{X}}
\nc\bfZ{\mathbf{Z}}
\nc\bfp{\mathbf{p}}
\nc\p{\mathbf{p}}
\nc\q{\mathbf{q}}
\nc\x{\mathbf{x}}

\nc\bbA{\mathbb{A}}
\nc\bbB{\mathbb{B}}
\nc\bbC{\mathbb{C}}
\nc\bbD{\mathbb{D}}
\nc\bbE{\mathbb{E}}
\nc\bbF{\mathbb{F}}
\nc\bbG{\mathbb{G}}
\nc\bbH{\mathbb{H}}
\nc\bbI{\mathbb{I}}
\nc{\bbJ}{\mathbb{J}} 
\nc\bbK{\mathbb{K}}
\nc\bbL{\mathbb{L}}
\nc\bbM{\mathbb{M}}
\nc\bbN{\mathbb{N}}
\nc\bbO{\mathbb{O}}
\nc\bbP{\mathbb{P}}
\nc\bbQ{\mathbb{Q}}
\nc\bbR{\mathbb{R}}
\nc\R{\mathbb{R}}
\nc\bbS{\mathbb{S}}
\nc\bbT{\mathbb{T}}
\nc\bbU{\mathbb{U}}
\nc\bbV{\mathbb{V}}
\nc\bbW{\mathbb{W}}
\nc\bbY{\mathbb{Y}}
\nc\bbX{\mathbb{X}}
\nc\bbZ{\mathbb{Z}}
\nc\Z{\mathbb{Z}}
\nc\cA{\mathcal{A}}
\nc\cB{\mathcal{B}}
\nc\cC{\mathcal{C}}
\rnc\cD{\mathcal{D}}
\nc\cE{\mathcal{E}}
\nc\cF{\mathcal{F}}
\nc\cG{\mathcal{G}}
\rnc\cH{\mathcal{H}}
\nc\cI{\mathcal{I}}
\nc{\cJ}{\mathcal{J}} 
\nc\cK{\mathcal{K}}
\rnc\cL{\mathcal{L}}
\nc\cM{\mathcal{M}}
\nc\cN{\mathcal{N}}
\nc\cO{\mathcal{O}}
\nc\cP{\mathcal{P}}
\nc\cQ{\mathcal{Q}}
\rnc\cR{\mathcal{R}}
\nc\cS{\mathcal{S}}
\nc\cT{\mathcal{T}}
\nc\cU{\mathcal{U}}
\nc\cV{\mathcal{V}}
\nc\cW{\mathcal{W}}
\nc\cY{\mathcal{Y}}
\nc\cX{\mathcal{X}}
\nc\cZ{\mathcal{Z}}
\nc\ff{\mathfrak{f}}
\nc\fg{\mathfrak{g}}
\nc\fh{\mathfrak{h}}

\DeclarePairedDelimiter\abs{\lvert}{\rvert}%
\DeclarePairedDelimiter\norm{\lVert}{\rVert}%

\makeatletter
\let\oldabs\abs
\def\abs{\@ifstar{\oldabs}{\oldabs*}}
\let\oldnorm\norm
\def\norm{\@ifstar{\oldnorm}{\oldnorm*}}
\makeatother

\nc\Weyl{\cW}
\nc\nba{\lip\beta,\alpha\rip}
\nc\nla{\lip\lambda,\alpha\rip}
\nc\nlia{\lip\lambda_i,\alpha\rip}
\nc\nlib{\lip\lambda_i,\beta\rip}
\nc\nsblia{\lip\sigma_\beta(\lambda_i),\alpha\rip}
\nc\typeA{A}
\nc\typeB{B}
\nc\typeC{C}
\nc\typeD{D}
\nc\typeE{E}
\nc\typeF{F}
\nc\typeG{G}
\nc{\charS}{X^{\ast}(S)}
\nc{\cocharS}{X_{\ast}(S)}
\nc{\Gm}{\mathbb{G}_m}
\nc{\der}{\mathrm{d}}
\nc{\Tan}{{T}}
\nc{\gM}{g_{\cM}}
\nc{\Qbar}{\overline{\bbQ}}
\nc{\bfHrm}{\bfH_{r,m}}
\nc{\bfTrm}{\bfT_{r,m}}
\nc{\bfNrm}{\bfN_{r,m}}
\nc{\Wrm}{\cW_{r,m}}
\nc{\Wtwo}{\cW_2}
\nc{\Wtwoo}{\cW_2^{+}}
\nc{\Wtwooo}{\cW_2'}
\nc{\eps}{\varepsilon}
\nc{\cone}{\cR}

\nc\oGamma{\overline{\Gamma}}

\nc{\dmo}{\DeclareMathOperator}
\rnc{\Re}{\operatorname{Re}}
\rnc{\Im}{\operatorname{Im}}

\nc{\spn}{\operatorname{span}}

\dmo{\rank}{rank}
\dmo{\diag}{diag}
\dmo{\End}{End}
\dmo{\Lie}{Lie}
\dmo{\Ima}{Im}

\dmo{\Jac}{Jac}
\dmo{\Id}{Id}
\dmo{\Ann}{Ann}
\dmo{\Area}{Area}
\dmo{\Isom}{Isom}
\dmo{\Tran}{Tran}
\dmo{\Gal}{Gal}

\title[Equidistribution and Dirichlet-improvable vectors]{Equidistribution in the space of 3-lattices and Dirichlet-improvable vectors on planar lines} 

\author[Kleinbock]{Dmitry~Kleinbock}
\address{Brandeis University, Goldsmith 207, Waltham, MA 02454-9110}
\email{kleinboc@brandeis.edu}
\author[de Saxc\'e]{Nicolas~de Saxc\'e}
\address{CNRS, Université Paris-Nord 13, Paris, France}
\email{desaxce@math.univ-paris13.fr}
\author[Shah]{Nimish~A.~Shah}
\address{The Ohio State University, Columbus, OH 43210}
\email{shah@math.osu.edu}
\author[Yang]{Pengyu~Yang}
\address{Department of Mathematics, ETH Zürich, Zürich, Switzerland}
\email{pengyu.yang@math.ethz.ch}

\thanks{The first-named author was supported by NSF grant  DMS-1900560. The third-named author was supported by NSF grant DMS-1700394.}

\begin{document}
	\begin{abstract}
		Let $X=\SL_3(\bbR)/\SL_3(\bbZ)$, and $g_t=\diag(e^{2t}, e^{-t}, e^{-t})$. Let $\nu$ denote the push-forward of the normalized Lebesgue measure on a segment of a straight line in the expanding horosphere of $\{g_t\}_{t>0}$, under the map $h\mapsto h\SL_3(\bbZ)$ from $\SL_3(\bbR)$ to $X$. We give explicit necessary and sufficient Diophantine conditions on the line for equidistribution of each of the following families of measures on $X$:	
		
		    (1) $g_t$-translates of $\nu$ as $t\to\infty$.
			
			(2) averages of $g_t$-translates of $\nu$ over $t\in[0,T]$ as $T\to\infty$.
			
			(3) $g_{t_i}$-translates of $\nu$ for some $t_i\to\infty$.

		We apply this dynamical result to show that {Lebesgue-almost every point on the planar line $y=ax+b$ is not Dirichlet-improvable if and only  if $(a,b)\notin\bbQ^2$}.	\end{abstract}
	
	\maketitle

	\setcounter{tocdepth}{1}

	\section{Introduction}
	
	\subsection{Equidistribution of expanding translates of curves}
	Let {$G=\SL_{n+1}(\bbR)$, $\Gamma=\SL_{n+1}(\bbZ)$} and $X=G/\Gamma$. Let $\mu_{X}$ denote the unique $G$-invariant probability measure on $X$. Let {$g_t=\diag(e^{nt}, e^{-t}, \dots,e^{-t})$}, so that the expanding horospherical subgroup {of $G$}  associated to $g_1$ is 
	\[U^+=\{g\in G : g_{-t}gg_t\to e, t\to{+}\infty \}=
	\left\{\begin{pmatrix}
			1 & * & \cdots & * \\
			& 1 & & \\
			&  & \ddots & \\
			& &  & 1
			\end{pmatrix}\right\} \cong \bbR^{n}.
			\]
	Let $x_0=e\Gamma\in X$. Using the Margulis thickening method {(see e.g.\ \cite{KM96})}, 
	one can show that the $g_t$-translates of the horosphere $U^+x_0$ get equidistributed in $X$. One may ask what happens if we replace the whole horosphere with a bounded piece of a real-analytic
	{submanifold} therein. We note that $X$ can be identified with the space of unimodular lattices in $\R^{n+1}$, hence numerous applications of dynamics on this space to Diophantine approximation, see e.g.\ \cite{Dan85, KM98, KW08}.
	
	If the analytic submanifold is \emph{non-degenerate}, i.e.\ {is not contained in a proper affine subspace},
	then a result of the third named author in \cite{Sha09Invention} tells us that equidistribution still holds. See also \cite{SY18} for a generalization to differentiable submanifolds. It is thus a natural question to ask for conditions for equidistribution of degenerate submanifolds, such as proper affine subspaces of $U^+$. One expects  these conditions to be expressed in terms of Diophantine properties of the affine subspaces.
	
	The main goal of this article is to {give a complete solution to this problem in the case $n = 2$, that is,} 
	study the case of straight lines {in $\bbR^2$. In this case we will show} that the dynamics is completely controlled by  {Diophantine conditions on} the parameters of the straight line, and will give criteria for different types of equidistribution phenomena.
	
	In what follows we will specialize to $n=2$; that is, let $G=\SL_{3}(\bbR)$, $\Gamma=\SL_{3}(\bbZ)$, and 
		$g_t=\diag(e^{2t}, e^{-t},e^{-t})$, so that the expanding horospherical subgroup {of $G$} associated to $g_1$ is 
		\[U^+=
		\left\{\begin{pmatrix}
			1 & * 
			& * \\
			& 1 &  \\
			&  & 1
			\end{pmatrix}\right\}\cong \bbR^2.
			\]
		As before, we let $x_0=e\Gamma\in X =G/\Gamma$. Note that, under the identification of $X$ with the space of unimodular lattices in $\mathbb{R}^{3}$, $x_0 $ corresponds to the standard lattice $\mathbb{Z}^3\subset \mathbb{R}^3$.

	Let $\Wtwo$ denote the set of vectors $(a,b)\in\bbR^2$ for which there exists $C>0$ such that the system of inequalities%
\footnote{Here and throughout the paper, the notation with a brace and several inequalities is used to indicate that \emph{all} of the inequalities hold simultaneously.}
	\begin{equation}\label{eq:Wtwo}
	\begin{cases}
	\abs{qb+p_1}\leq C\abs{q}^{-2} \\
	\abs{qa+p_2}\leq C\abs{q}^{-2}
	\end{cases}
	\end{equation} 
	has infinitely many solutions $(p_1,p_2;q)\in\bbZ^2\times\bbN$.

	Similarly, let $\Wtwooo$ denote the set of vectors for which \eqref{eq:Wtwo} has a non-zero solution $(p_1,p_2;q)\in\bbZ^2\times\bbN$ for every $C>0$.
	
\subsubsection{Remark} \label{rem:W2W2'}One has an obvious inclusion $\Wtwooo \subset \Wtwo$.
	It can be deduced from \cite[Theorem 1.3]{Roy} that this inclusion is strict, even though both sets have Hausdorff dimension equal to $1$, see \cite{Dod92}.
	
\bigskip	
Throughout the paper, $I=[s_0,s_1]$ denotes an arbitrary compact interval with non-empty interior, i.e. $s_0<s_1$.
	For  $(a,b)\in\bbR^2$, let $\phi_{a,b}:I\to U^+$ be the line segment defined by $$\phi_{a,b}(s)= {\begin{pmatrix}
		1 & s & as+b \\
		& 1 & \\
		& & 1
		\end{pmatrix}}, \ \forall s\in I.$$
	Let $\lambda_{a,b}$ denote the push-forward of the normalized Lebesgue measure on $I$ under the map $s\mapsto\phi_{a,b}(s)x_0$ from $I$ to 
   $X$, and for any $t\in\R$, let $g_t\lambda_{a,b}$ denote the translate of $\lambda_{a,b}$ by $g_t$; that is, for any $f\in C_c(X)$,
	\begin{equation} \label{eq:lambda_ab}
	\int_{X} f\,\mathrm{d}\lambda_{a,b}=\fint_I f(\phi_{a,b}(s)x_0)\,\mathrm{d}s:=\frac{1}{\abs{I}}\int_{I} f(\phi_{a,b}(s)x_0)\,\mathrm{d}s,
	\end{equation}
	\begin{equation}\label{eq:gtlambda_ab}
	\int_{X} f\, \mathrm{d}(g_t\lambda_{a,b})=\fint_I f(g_t\phi_{a,b}(s)x_0)\,\mathrm{d}s,
	\end{equation}
	where $\abs{A}$ is the Lebesgue measure of $A$ for any measurable subset $A$ of $\bbR$.

We say that a family $\{\lambda_i\}_{i\in\cI}$ of probability measures on $X$ has \eb{no escape of mass\/} if for every $\eps>0$ there exists a compact subset $K$ of $X$ such that $\lambda_i(K)>1-\eps$ for all $i\in\cI$.

	We will prove the following criterion for non-escape of mass.
	\begin{thm}\label{thm:main_thm_non_escape_of_mass}
		The translates $\{g_t\lambda_{a,b}\}_{t>0}$ have no escape of mass if and only if 
		\[
		(a,b)\notin \Wtwooo.
		\]
	\end{thm}
	
	Furthermore, for $\cI=\bbN$ or $\bbR_{>0}$, we say that a family $\{\lambda_i\}_{i\in\cI}$ of probability measures on $X$ gets \eb{equidistributed} in $X$ if
	\begin{equation*}
	\int f \,\mathrm{d}\lambda_i\stackrel{i\to\infty}{\longrightarrow}\int f\,\mathrm{d}\mu_{X},\;\forall f\in C_c(X);
	\end{equation*}
	that is, $\lambda_i$ converges to $\mu_{X}$ with respect to the weak-* topology as $i\to\infty$.

	\begin{thm}\label{thm:main_thm_equidistribution}
		The translates $\{g_t\lambda_{a,b}\}_{t>0}$ get {equidistributed} in $X$ if and only if 
		\[
		(a,b)\notin \Wtwo.
		\]
	\end{thm}
	
	Since the set $\Wtwo$ has Lebesgue measure $0$ in $\bbR^2$, a typical line gets equidistributed under the flow $g_t$.
	
	Chow and Yang \cite{CY19} proved effective equidistribution for translates of a Diophantine line by diagonal elements near $\diag(e^t, e^{-t}, 1)$. Unfortunately their method does not seem to apply to the flow $g_t=\diag(e^{2t},e^{-t},e^{-t})$ here. We will instead use Ratner's measure rigidity theorem for unipotent flows \cite{Ra91}, and tools from geometric invariant theory.

	\subsection{Averaging over the time parameter}
	Define
	\begin{equation*}
	\Wtwoo=\left\{(a,b)\in\bbR^2 : \limsup_{(p_1,p_2;q)\in\bbZ^2\times\bbN} \frac{-\log \max\{\abs{qb+p_1},\abs{qa+p_2}\}}{\log {q}}> 2 \right\}.
	\end{equation*}
	{In other words}, $\Wtwoo$ consists of vectors $(a,b)\in\bbR^2$ for which there exists $\eps>0$ such that the system of inequalities
	\begin{equation*}\label{eq:Wtwoo1}
	\begin{cases}
	\abs{qb+p_1}\leq {q}^{-(2+\eps)} \\
	\abs{qa+p_2}\leq {q}^{-(2+\eps)}
	\end{cases}
	\end{equation*} 
	has infinitely many solutions $(p_1,p_2;q)\in\bbZ^2\times\bbN$.
	
	\subsubsection{Remark} In view of Remark~\ref{rem:W2W2'}, we have strict inclusions
\[ \Wtwoo\subsetneq\Wtwooo\subsetneq\Wtwo ,\]
	even though all of these sets have Hausdorff dimension equal to $1$ (see \cite{Dod92}). The strictness of the inclusion  $\Wtwoo\subset\Wtwooo$ can be derived from a zero--infinity law for Hausdorff measures of those sets with appropriate dimension functions, see \cite{DickinsonVelani}.
	
	\bigskip
	We are also interested in the limit distributions of the \eb{averages} of $g_t$-translates of $\lambda_{a,b}$, namely the family $\left\{\frac{1}{T}\int_{0}^{T}g_t\lambda_{a,b}\,\mathrm{d}t\right\}_{T>0}$ of probability measures on $X$. 
	Similar questions have been considered in \cite{SW17} and from the measure rigidity point of view in \cite{ES19}.

	\begin{thm}\label{thm:main_thm_average}
		The following are equivalent:
		
		\begin{enumerate}
			\item[\rm (1)] The averages $\{\frac{1}{T}\int_{0}^{T}g_t\lambda_{a,b}\,\mathrm{d}t\}_{T>0}$  get equidistributed in $X$.
			
			\item[\rm (2)] The averages $\{\frac{1}{T}\int_{0}^{T}g_t\lambda_{a,b}\,\mathrm{d}t\}_{T>0}$ have no escape of mass.
			
			\item[\rm (3)]
			$
			(a,b)\notin\Wtwoo.
			$\end{enumerate}
	\end{thm}
	
	\subsubsection{Remark}
		It is shown in \cite{Kle03} that the planar line $\{y=ax+b\}$ is \emph{extremal} if and only if $(a,b)\notin\Wtwoo$.

	\subsection{Equidistribution along a sequence}
	We are also interested in understanding when $\{g_t\lambda_{a,b}\}_{t>0}$ equidistributes along some subsequence $t_i\to\infty$.
	
	\begin{thm}\label{thm:main_thm_sequence}
		Let $(a,b)\in\bbR^2$. Then the following are equivalent:
		\begin{enumerate}
			\item[\rm (1)] $(a,b)\notin\bbQ^2$.
			\item[\rm (2)] 
			The closure of $\{g_t\lambda_{a,b}\}_{t\geq 0}$ contains $\mu_X$ with respect to the weak-* topology.
			\item[\rm (3)]
			For almost every $s\in \bbR$, the trajectory 
			$
			\{ g_t\phi_{a,b}(s)x_0 \}_{t\geq 0}
			$
			is dense in $X$.
		\end{enumerate}
	\end{thm}
	
	\subsubsection{Remark}
		Suppose $(a, b)\in\bbQ^2$, then $g_t\lambda_{a,b}$ will diverge, i.e.\ eventually leave any fixed compact set, see Remark~\ref{rem:Q-diverge}. Hence \Cref{thm:main_thm_sequence} gives us a dichotomy which was somewhat unexpected: the $g_t\lambda_{a,b}$ \emph{either} diverge as $t\to\infty$, \emph{or} get equidistributed along some sequence $t_i\to\infty$.
		
		\subsubsection{Remark} \label{rmk:main_thm_equidistribution_dual}
		We also have dual versions of \Cref{thm:main_thm_non_escape_of_mass}, \Cref{thm:main_thm_equidistribution}, \Cref{thm:main_thm_average} and \Cref{thm:main_thm_sequence} above. Let $\tilde{g_t}=\diag(e^t,e^t,e^{-2t})$, consider the map  $\tilde{\phi}_{a,b}\colon s\mapsto {\begin{pmatrix}
			1 &  & as+b \\
			& 1 & s\\
			& & 1
			\end{pmatrix}}$, and let $\tilde\lambda_{a,b}$ denote the push-forward of the normalized Lebes\-gue measure on $I$ under the map $s\mapsto\tilde\phi_{a,b}(s)x_0$ from $I$ to 
   $X$. 
		Then all the four theorems are still valid for $\tilde{g_t}$ in place of $g_t$ and $\tilde{\lambda}_{a,b}$ in place of $\lambda_{a,b}$. 
		Indeed, it suffices to consider the outer automorphism $g\mapsto w\cdot{}^tg^{-1}\cdot w$ of $G$, where 
		$w={\begin{pmatrix}
			& & 1 \\
			& -1 & \\
			1 & &
			\end{pmatrix}}$; under this automorphism $g_t$ is sent to $\tilde{g_t}$, $\lambda_{a,b}$ is sent to $\tilde{\lambda}_{a,b}$, and $\mu_X$ is preserved. See \cite[Page 511]{Sha09Invention}.

	\subsubsection{Remark} \label{rmk:main_thm_indep_of_I} It is also worthwhile to point out that even though the measures $\lambda_{a,b}$ depend on the choice of   $I = [s_0,s_1]$, the criteria in all the theorems stated above do not; that is, the limiting behavior of these measures is the same for all nontrivial intervals simultaneously. Using the arguments of this article, one can see that for a given sequence $t_i\to\infty$, if $g_{t_i}\lambda_{a,b}\to\mu_X$, then for every finite interval $J$ with nonempty interior, we have $g_{t_i}\lambda_{a,b}^J\to\mu_X$.

	\subsection{Dirichlet-improvable vectors on planar lines}
	The motivation for our study came from Diophantine approximation. {Denote by $\norm{\cdot}$ the supremum norm on $\R^n$ (unless specified otherwise, all the norms on $\bbR^n$ will be taken to be the supremum norm).} {Following} Davenport and Schmidt \cite{DS6970}, for $0<\delta<1$ we say that a vector 
	{$\x\in\R^n$} is \eb{$\delta$-improvable}
	if for every sufficiently large $T$, the system of inequalities
	{\begin{equation*}\label{vect}
		\begin{cases}
		{\|q\x+\p\|}\leq \delta T^{-1} \\
		\abs{q} \leq  T^n
		\end{cases}
		\end{equation*}
		has a solution 
		$(\p,q)$, where $\p\in\Z^n$ and $q\in\Z\setminus\{0\}$}. One says that {$\x$} is \eb{Dirichlet-improvable} if it is $\delta$-improvable for some $0<\delta<1$, {and that it is \eb{singular} if it is $\delta$-improvable for all $0<\delta<1$}.
	
	Similarly, a real linear form {on $\q\in \bbZ^n$ is given by $\q\mapsto \x\cdot\q$, parametrized by $\x\in\bbR^n$}. We say that this linear form is \eb{$\delta$-improvable} if there exists $0<\delta<1$ such that for every sufficiently large $T$, the system of inequalities
	{\begin{equation}\label{lf}
		\begin{cases}
		\abs{\x\cdot\q+p}\leq \delta T^{-n} \\
		\|\q\| \leq  T\end{cases}
		\end{equation}
		has a solution 
		$(p,\q)$, where $p\in\Z$ and $\q\in\Z^n\setminus\{0\}$}. 
	
	The notation {$\mathrm{DI}(n,1)$ and $\mathrm{DI}(1,n)$ (resp., $\mathrm{Sing}(n,1)$ and $\mathrm{Sing}(1,n)$)} is used in the literature to denote the set of Dirichlet-improvable (resp., singular) vectors and linear forms. It is 
	{known that $\mathrm{DI}(n,1)=\mathrm{DI}(1,n)$ and $\mathrm{Sing}(n,1) = \mathrm{Sing}(1,n)$, see \cite{DS70}} and \cite[Chapter V, Theorem XII]{Cassels57} respectively.
	
	The readers who would like to know more background information are referred to \cite{KW08,Sha09Invention} and references therein for Dirichlet-improvable vectors, and \cite{CC16,Dan85} and references therein for singular vectors.
	
	{Now let us again specialize to $n=2$, and take $\x$ of the form $(s,as+b)$. In the simplest possible case $(a,b)\in\bbQ^2$ it is very easy to see that every point on the planar line \[
		L_{a,b}=\{ (x,y)\in\bbR^2 : y = ax+b \}
		\] 
		is singular: indeed, take $a = k/m$ and $b = \ell/m$ and notice that one has $$(x,y)\cdot (-k,m) = \left(s,\tfrac kms+\tfrac \ell m\right)\cdot (-k,m)  = \ell;$$
		thus one can always find $(p,\q)$ such that the left hand side of the first inequality in \eqref{lf} iz zero.}	
	
	Our next main theorem is the following {stronger converse to the above computation}:
	
	\begin{thm}\label{thm:main_thm_Dirichlet}
		Let $(a,b)\in\bbR^2$. If $(a,b)\notin\bbQ^2$, then almost every point on the planar line $L_{a,b}$
		is not Dirichlet-improvable.
	\end{thm}
	
	The deduction of \Cref{thm:main_thm_Dirichlet} from \Cref{thm:main_thm_sequence} uses Dani's correspondence, and has become a standard argument. We give the proof below for completeness.
	
	\begin{proof}[Proof of \Cref{thm:main_thm_Dirichlet} assuming \Cref{thm:main_thm_sequence}]
		For $0<\delta<1$, let $K_\delta\subset X$ denote the set of unimodular lattices in $\bbR^3$ whose shortest non-zero vector has norm at least $\delta$. Then $K_\delta$ contains an open neighborhood of $x_0$, and is compact by Mahler's compactness criterion. For $0<\delta<1$, let $D_\delta$ denote the set of $s\in\bbR$ such that $(s,as+b)$ is a $\delta$-improvable linear form.
		By Dani's correspondence \cite{Dan85} and \cite[Proposition 2.1]{KW08},
		\[
		D_\delta=\{ s\in\bbR : g_t\phi_{a,b}(s)x_0\notin K_{\delta^{1/3}} \text{ for all large } t \}.
		\]
		In particular, $\{ g_t\phi_{a,b}(s)x_0 \}_{t\geq 0}$ is not dense in $G/\Gamma$ for all $s\in D_\delta$. By \Cref{thm:main_thm_sequence} we have $\abs{D_\delta}=0$ for all $0<\delta<1$. Finally, we conclude the proof by noting that the set of $s$ such that $(s, as+b)$ is a Dirichlet-improvable point on $L_{a,b}$ equals $\bigcup_{m\geq1}D_{\frac{m-1}{m}}$, and hence has Lebesgue measure 0.
	\end{proof}
	
	\subsection{Strategy of the proof} Given any sequence $t_i\to\infty$, after passing to a subsequence we obtain that $g_{t_i}\lambda_{a,b}$ converges to a measure, say $\mu$, on $X$ with respect to the weak-* topology. It is straightforward to see that $\mu$ must be invariant under a non-trivial unipotent subgroup of $G$ (\Cref{prop:U-inv}). There are two possibilities: if $\mu$ is not a probability measure then we apply the Dani-Margulis non-divergence criterion (\Cref{thm:KMnondivergence}), and if $\mu$ is positive and not $G$-invariant then  we apply Ratner's description of ergodic invariant measures for unipotent flows, combined with the linearization technique (\Cref{prop:consequence_of_linear_focusing}). In both  cases we obtain the following {\em{linear dynamical obstruction to equidistribution:\/}} There exist a finite-dimensional representation $V$ of $G$ over $\bbQ$, a non-zero vector $v_0\in V(\bbQ)$, a constant $R>0$, and a sequence $\{\gamma_i\}\subset \Gamma=\SL_3(\bbZ)$ such that for each $i$,
	\begin{equation} \label{eq:basic-obstruction}
	\sup_{s\in I} \norm{g_{t_i}\phi_{a,b}(s)\gamma_iv_0}\leq R.
	\end{equation}
	
	The major effort involved in this proof is to analyze this linear dynamical obstruction and show that $(a,b)$ must satisfy certain Diophantine approximation condition. 
	
	Using Kempf's numerical criterion in geometric invariant theory, when the $Gv_0$ is not Zariski closed, 
	we reduce the obstruction to the case of $v_0$ being a highest weight vector (\Cref{prop:reduction_to_eigenvector}).
	Then we further reduce to the case of $v_0$ being a highest weight vector of a fundamental representation of $G$,
	namely the standard representation $\bbR^3$, or its exterior power $\bigwedge^2 \bbR^3$ (\Cref{lem:reduction_to_fundamental_rep}).
	It is straightforward to show that the obstruction \eqref{eq:basic-obstruction} does not arise for the exterior representation (\Cref{lem:exterior_square}), 
	so we are left only with the case of $V$ being the standard representation. 
	In the case of the standard representation the dynamical obstruction leads to the Diophantine condition that $(a,b)\in \Wtwo$ (\Cref{lem:interpretation_of_Wtwo}). 

    We are left with the case of $Gv_0$ being Zariski closed. Using explicit descriptions of finite-dimensional irreducible representations of $\SL_2$ and $\SL_3$, we show that in this case after passing to a further subsequence $\{\gamma_i v_0\}$ is constant and $(a,b)\in\bbQ^2$ (\Cref{prop:consequence_of_linear_focusing}). 
    
    We remark that for $G=\SL_n$ for $n>3$, analyzing the Zariski closed orbit case involves much greater complexities, and the above strong conclusion about $(a,b)$ is not possible. 

	\subsection{Comparison with previous work}
	We remark that \Cref{thm:main_thm_average} and \Cref{thm:main_thm_Dirichlet} sharpen the main results of Shi and Weiss \cite{SW17}. More precisely, it was shown in \cite{SW17} that the \emph{averages} of $g_t$-translates of $\lambda_{a,b}$ get {equidistributed} in $X$ if the line $\{y = ax+b\}$ contains a badly approximable vector.
By Remark~\ref{subsec:badly} below, this condition implies in particular that $(a,b)\notin \Wtwooo$, and so $(a,b)\notin\Wtwoo$.
Our results give sharp conditions of non-escape of mass and equidistribution for not only averages, but also pure translates which was not considered in \cite{SW17}.
This is why we are able to prove the much stronger \Cref{thm:main_thm_Dirichlet}.

	\subsection{Future directions}
	Instead of $g_t$, one may consider more general flows. It seems that our method is also applicable to the study of translates by elements in a Weyl chamber, at least for certain cones. Then the non-effective version of Theorem 1.1 of \cite{CY19} will be recovered. One would also be able to say something about improvability of weighted Dirichlet Theorem, and the readers are referred to \cite{Sha10} for a detailed introduction to this subject.
	
	One may also ask what happens to other Lie groups, e.g. $G=\SL_n(\R)$ for $n>3$. When $n=4$, things already become more complicated. Roughly speaking, $\SL_3(\R)$ is small and one does not have many choices of possible intermediate subgroups. However, in $\SL_n(\R)$, where $n>3$, there are more possibilities of intermediate subgroups; see the follow-up paper \cite{NimishPengyu} for more details. 
	
	\subsection{Acknowledgements}
	We would like to thank Emmanuel Breuillard, Alexander Gorodnik and Lior Silberman for helpful discussions. Part of the work was done when the second and the fourth-named authors were visiting the Hausdorff Research Institute for Mathematics (HIM) in Bonn for the trimester program ``Dynamics: Topology and Numbers'' in 2020; they would like to thank HIM for hospitality. We also thank the referees for their very careful detailed comments and corrections that greatly helped us improve the readability of the paper.

	\section{Instability and invariant theory}\label{sect:2}

	To analyze limiting distributions of sequences of translates of measures on homogeneous spaces a technique has been developed, where one applies the Dani-Margulis and Kleinbock-Margulis non-divergence criteria, Ratner's theorem and the linearization method, and reduces the problem to dynamics of subgroup actions on finite-dimensional representations of semisimple groups, see ~\cite{Sha09Invention,SY16}. In \cite{Yan20Invent}, this kind of linear dynamics was analysed in a very general situation using invariant theory results due to Kempf~\cite{Kem78}. We follow the same approach, and this section is devoted to describing the basic tools from geometric invariant theory that we shall need in our argument.

\bigskip
Let $G$ be a reductive real algebraic group defined over $\bbQ$, and $\rho: G \to \GL(V)$ a linear representation of $G$ defined over $\bbQ$.
	We say that a nonzero vector $v\in V$ is \eb{unstable} if the Zariski closure of the orbit $G v$ contains the origin. 
	Hilbert-Mumford's unstability criterion states that a nonzero vector $v$ is unstable if and only if there exists a cocharacter $\lambda\colon\mathbb{G}_m\to G$ such that $\lambda(t)v\stackrel{t\to0}{\longrightarrow}0$.
	Kempf~\cite{Kem78} refined this criterion by studying the set of cocharacters, up to scaling, $\lambda$ such that $\lambda(t)$ bring $v$ to $0$ at maximal speed as $t\to0$. Let us briefly recall his results.

\smallskip
	Write $X_{\ast}(G)$ for the set of $\bbQ$-cocharacters of $G$.
	For any nonzero $v\in V(\bbQ)$ and any nontrivial cocharacter $\lambda$ in $X_{\ast}(G)$, one can write $v=\sum_{i\in\bbZ}v_i$, where $\lambda(t)v_i=t^iv_i$ for all $i$.
	Let $m(v,\lambda)=\min\{i\in\bbZ:v_i\neq 0\}$. Then 
	\[
	v=v_{m(v,\lambda)}+\sum_{i>m(v,\lambda)} v_i.
	\] 
	Thus for any $g\in G(\bbQ)$, $m(gv,g\lambda g^{-1})=m(v,\lambda)$.
	
	For any $\lambda\in X_\ast(G)$, the group
	\[
	P(\lambda)=\left\{p\in G : \lim_{t\to 0} \lambda(t)p\lambda(t)^{-1} \text{ exists in } G \right\}
	\]
	is a parabolic subgroup of $G$ defined over $\bbQ$, and 
	\[
	R_u(P(\lambda))=\left\{u\in G : \lim_{t\to 0} \lambda(t)u\lambda(t)^{-1}=e\right\}
	\]
	is the unipotent radical of $P(\lambda)$ defined over $\bbQ$. Also $P(\lambda)=Z_G(\lambda)R_u(P(\lambda))$, and this product holds over $\bbQ$-points, where $Z_G(\lambda)$ is the centralizer of the image of $\lambda$ in $G$. 
	
	We note that if $u\in R_u(P(\lambda))$, then 
	\begin{equation}
	\label{rem:U+}
	uv=v_{m(v,\lambda)}+\sum_{i>m(v,\lambda)} (uv)_i.
	\end{equation}
	
	Let $S$ be a maximal $\bbQ$-split torus in $G$, we fix a positive definite integral bilinear form $(\cdot,\cdot)$ on the free abelian group $\cocharS$ of $\bbQ$-cocharacters on $S$ which is invariant under the Weyl group $N_G(S)/Z_G(S)$; it induces a norm on $\cocharS$ defined by $\norm{\lambda}=\sqrt{(\lambda,\lambda)}$.
	This norm extends uniquely to a norm on the set $X_\ast(G)$ of $\bbQ$-cocharacters of $G$ which is invariant under the conjugation by $G(\bbQ)$.
	
	\subsection*{Kempf's Theorem~\cite[Theorem~4.2]{Kem78}} {\em
	Let $v\in V(\bbQ)$ be a nonzero unstable vector. Then the following hold:
	
	{\rm a)} Let $B_v=\sup\bigl\{{m(v,\lambda)}/{\norm{\lambda}}: \lambda\in X_\ast(G) \text{ nontrivial}\bigr\}$. Then $B_v>0$.
	
	{\rm b)} Let $\Lambda_v=$ the set of indivisible $\lambda\in X_\ast(G)$ such that $m(v,\lambda)=B_v\cdot \norm{\lambda}$. Then,
	\begin{enumerate}
	\item[\rm (1)] \label{itm:Lambda_v} $\Lambda_v$ is non-empty.
	\item[\rm (2)] \label{itm:Pv}
	There exists a $\bbQ$-parabolic subgroup $P_v$ of $G$ such that $P_v=P(\lambda)$ for all $\lambda\in \Lambda_v$.
	\item[\rm (3)] \label{itm:conj} The set $\Lambda_v$ is a principle homogeneous space under conjugation by $\bbQ$-points of the unipotent radical of $P_v$. In particular, $P_v(\bbQ)$ acts transitively on $\Lambda_v$ under conjugation.
	\item[\rm (4)] For any maximal torus of $P_v$, which is defined over $\bbQ$, contains the image of a unique member of $\Lambda_v$.
	\end{enumerate}}
	
	In the above result we observe that for any $g\in G(\bbQ)$, $gv$ is also unstable, $B_v=B_{gv}$, $\Lambda_{gv}=g \Lambda_v g^{-1}$, and $P_{gv}=gP_{v}g^{-1}$. Therefore, if $g\in P_v(\bbQ)$, then by (3) of b) above, $\Lambda_{gv}=\Lambda_{v}$, and in particular, $m(gv,\lambda)=m(v,\lambda)$ for all $\lambda\in \Lambda_{v}$. 

	\bigskip
	We will apply Kempf's theorem to the group $G=\SL_3(\bbR)$.
	In that case, the maximal torus $S$ is chosen to be the subgroup of $G=\SL_3(\bbR)$ consisting of diagonal matrices. Then $\delta\in \cocharS$ means that there exists a unique $(a,b,c)\in\bbZ^3$ such that $a+b+c=0$ and $\delta(t)=\diag(t^a,t^b,t^c)$ for all $t\neq 0$. The Euclidean inner-product on $\bbZ^3$ restricts to a Weyl group invariant inner-product on $\cocharS$ with respect to this identification. So for $\delta$ as above, $\norm{\delta}=\sqrt{a^2+b^2+c^2}$.  
	Let $\charS$, denote the abelian group of $\bbQ$-characters on $S$. Define a bilinear pairing $\langle\cdot,\cdot\rangle:\charS\times\cocharS\to \bbZ$ of the $\bbZ$-modules such that for any $\chi\in\charS$ and $\delta\in\cocharS$, we have $\chi(\delta(t))=t^{\langle \chi,\delta\rangle}$ for all $t\neq 0$. Let $\delta^\vee\in \charS$ denote the dual to $\delta$ in the following sense: $\langle\delta^\vee,\lambda\rangle=(\delta,\lambda)$ for all $\lambda\in \cocharS$. So for the $\delta$ described as above, $\delta^\vee(\diag(t_1,t_2,t_3))=t_1^at_2^bt_3^c$ for all $\diag(t_1,t_2,t_3)\in S$.

	Choose simple roots 
	\[
	\alpha_1\colon \diag(t_1,t_2,t_3)\mapsto t_1t_2^{-1} \text{ and } \alpha_2\colon \diag(t_1,t_2,t_3)\mapsto t_2t_3^{-1}.
	\]
	The corresponding fundamental weights are 
	\[
	\omega_1\colon \diag(t_1,t_2,t_3)\mapsto t_1 \text{ and } \omega_2\colon \diag(t_1,t_2,t_3)\mapsto t_1t_2.
	\]
	
	For any non-negative integers $n_1$ and $n_2$, there exists a unique irreducible representation of $G$ with highest weight $n_1\omega_1+n_2\omega_2$, where we use the additive notation for $\charS$ (see \cite{FultonHarris}[Theorem~13.1]). 
	
	The standard parabolic subgroups of $G$ are
	\[
	P_0=\left\{\begin{pmatrix}
	* & * & * \\
	& * & * \\
	& & *
	\end{pmatrix}\right\},\quad
	P_1=\left\{\begin{pmatrix}
	* & * & * \\
	& * & * \\
	& * & *
	\end{pmatrix}\right\},\quad 
	P_2=\left\{\begin{pmatrix}
	* & * & * \\
	* & * & * \\
	& & *
	\end{pmatrix}\right\}.
	\]
	We shall also use the following algebraic subgroups of $G$:
	\[ 
	Q_0=\left\{\begin{pmatrix}
	1 & * & * \\
	& 1 & * \\
	& & 1
	\end{pmatrix}\right\},\quad
	Q_1=\left\{\begin{pmatrix}
	1 & * & * \\
	& * & * \\
	& * & *
	\end{pmatrix}\right\},\quad
	Q_2=\left\{\begin{pmatrix}
	* & * & * \\
	* & * & * \\
	& & 1
	\end{pmatrix}\right\}.
	\]
	\[
	S_0=\left\{\begin{pmatrix}
	* &  &  \\
	& * &  \\
	& & *
	\end{pmatrix}\right\},\quad
	S_1=\left\{\begin{pmatrix}
	t^{2} &  &  \\
	& t^{-1} &  \\
	&  & t^{-1}
	\end{pmatrix}\right\},\quad
	S_2=\left\{\begin{pmatrix}
	t &  &  \\
	& t &  \\
	& & t^{-2}
	\end{pmatrix}\right\}.
	\]
	\[ 
	H_0=\{1\},\quad
	H_1=\left\{\begin{pmatrix}
	1 &  &  \\
	& * & * \\
	& * & *
	\end{pmatrix}\right\},\quad
	H_2=\left\{\begin{pmatrix}
	* & * &  \\
	* & * &  \\
	& & 1
	\end{pmatrix}\right\}.
	\]
	\[ 
	U_0=\left\{\begin{pmatrix}
	1 & * & * \\
	& 1 & * \\
	& & 1
	\end{pmatrix}\right\},\quad
	U_1=\left\{\begin{pmatrix}
	1 & * & * \\
	& 1 &  \\
	&  & 1
	\end{pmatrix}\right\},\quad
	U_2=\left\{\begin{pmatrix}
	1 &  & * \\
	& 1 & * \\
	& & 1
	\end{pmatrix}\right\}.
	\]
	One has $P_i=S_iQ_i=S_iH_iU_i$, for $i=0,1,2$.

\subsection{Reduction to a highest weight vector}
The next result (\Cref{prop:reduction_to_eigenvector}) provides a powerful new technique that allows one to reduce the study of linear dynamics of an arbitrary vector in an arbitrary representation to that of a highest weight vector. From this we will further reduce the study to fundamental representations (\Cref{lem:reduction_to_fundamental_rep}), opening the doors to directly relate the linear dynamics to Diophantine properties of vectors (\Cref{lem:interpretation_of_Wtwo}). 
	
	The proof of the following result was motivated by \cite[Proposition~2.4]{Yan20Invent}. 
	
	Throughout this article, we will assume that all the finite dimensional vector spaces are equipped with some norm, denoted by $\norm{\cdot}$.
	
	\begin{thm}\label{prop:reduction_to_eigenvector}
		Let $V$ be a representation of $G$ defined over $\bbQ$. Let $v$ be an unstable vector in $V(\bbQ)$. Then there exists an irreducible representation $W$ of $G$ defined over $\bbQ$, a highest weight vector $w'\in W(\bbQ)$, an element $g_0\in G(\bbQ)$, a real number $\beta>0$, and a real number $C>0$ such that  for any $g\in G$ one has
		\begin{equation*}\label{eq:temp1}
		\lVert gg_0w' \rVert \leq C \lVert gv \rVert^\beta.
		\end{equation*}
	\end{thm}
	
	\begin{proof}
		Without loss of generality, we may assume that the norms are $K:=\SO(3)$-invariant. Given the unstable $v\in V(\bbQ)$, let $B_v>0$, $\Lambda_v\subset X\ast(G)$, and a $\bbQ$-parabolic subgroup $P_v$ of $G$ be as given by Kempf's theorem. 
	By \cite[Proposition 21.12]{Bor91}, there exist $g_0\in G(\bbQ)$ and $j\in\{0,1,2\}$ such that $P_{v} = g_0P_jg_0^{-1}$.
	Let $v'=g_0^{-1}v$. Then $v'\in V(\bbQ)$ is also unstable, and by Kempf's theorem, $P_{v'}=g_0^{-1}P_vg_0=P_j$, and since $S\subset P_j=P_{v'}$, we have that $\cocharS$ contains a unique member, say $\delta$, of $\Lambda_{v'}$.  Therefore 
		\begin{equation*} 
		    P(\delta)=P_{v'} = P_j.
		\end{equation*}
		Hence, $\Ima\delta$ is contained in $S_j$, and $\delta(t)=\diag(t^a,t^b,t^c)$ for all $t\neq 0$ such that $(a,b,c)\in\bbZ^3$, $a+b+c=0$ and $a\geq b\geq c$. 
		
		Now let $\delta^{\vee}\in \charS$ be dual to $\delta$. 
		Let $S_\delta$ denote the $\bbQ$-subtorus of $S_j$ which is the identity component of the kernel of $\delta^\vee$ in $S_j$. We have that $\Ima\delta\cap S_\delta$ is finite and $(\Ima \delta)S_\delta=S_j$.
		
		We have that $\delta^\vee=(a-b)\omega_1+(b-c)\omega_2$, with $a-b\geq 0$ and $b-c\geq 0$. Therefore there exists an irreducible representation $W$ of $G$ defined over $\bbQ$ with the highest weight $\delta^\vee$. Let $w'\in W(\bbQ)$ be a highest weight vector. Let 
		\[
		\beta=\frac{(\delta,\delta)}{m(v',\delta)}=\frac{1}{B_{v'}}>0.
		\]
		
		Now it suffices to show that there exists $C>0$ such that for any $g\in G$ we have
		\begin{equation*}
		\norm{gw'}\leq C\norm{gv'}^\beta.
		\end{equation*}
		To argue by contradiction, suppose that there exists a sequence $\{g_i\}\subset G$ such that 
		\begin{equation*} 
		\lim_{i\to\infty} \frac{\lVert g_iv' \rVert^\beta }{ \lVert g_iw' \rVert}=0. 
		\end{equation*}
		
		We note that $P(\delta)=P_j=S_jQ_j$, $Q_j=H_jU_j$ fixes $w'$, and $S_j$ acts on $w'$ via the character $\delta^{\vee}$. Since $G=KP_j=KS_jH_jU_j$, we can write $g_i=k_is_ih_iu_i$ where $k_i\in K$, $s_i\in S_j$, $h_i\in H_j$ and $u_i\in U_j$. Since the norms are $K$-invariant, we may assume that $k_i=e$ for all $i$. Now  $s_ih_iu_iw'=\delta^{\vee}(s_i)w'$. Hence $\lVert g_iw' \rVert=\lvert\delta^{\vee}(s_i)\rvert\lVert w' \rVert$. Since $S_j=(\Ima\delta)S_\delta$, we can write $s_i=\delta(\tau_i)\sigma_i$, where $\tau_i\in\bbR^\times$ and $\sigma_i\in S_\delta$. Then 
		\begin{equation}\label{eq:temp2}
		\lVert g_iw' \rVert=\lvert\delta^\vee(\delta(\tau_i))\rvert\lVert w' \rVert=\abs{\tau_i}^{(\delta,\delta)}\norm{w'}.
		\end{equation} 
		
		We consider the weight space decomposition $V = \oplus V_{\chi}$, where $S$ acts on $V_{\chi}$ by multiplication via the character $\chi$ of $S$, where each $V_{\chi}$ is defined over $\bbQ$ as $S$ is a $\bbQ$-split torus. Let 
		\begin{equation*} 
		\tilde{V}=\{ x\in V : \delta(t)x=t^{m(v',\delta)} x \}=\oplus\{V_\chi: \chi\in\charS \text{ and }\langle\chi,\delta\rangle=m(v',\delta)\}.
		\end{equation*}
		Let $\pi:V\to \tilde{V}$ denote the natural projection defined over $\bbQ$. Then $\pi(v')\in \tilde{V}(\bbQ)$. Since $\delta\in \Lambda_{v'}$, we have that $\pi(v')=v'_{m(v',\delta)}\neq 0$. Since $S_jH_j$ is contained in the centralizer of $\delta$, we have that $\pi$ is $S_jH_j$-equivariant.
		
		There exists $C_1>0$ such that $\lVert \pi(x) \rVert \leq C_1\lVert x \rVert$ for all $x\in V$. It follows that
		\begin{equation}\label{eq:temp3}
		\frac{\lVert \pi(g_iv')\rVert^\beta}{\lVert g_iw' \rVert} \leq C_1^\beta \frac{\norm{g_iv'}^\beta}{\norm{g_iw'}}\underset{i\to\infty}{\longrightarrow} 0.
		\end{equation}
		
		For any $u\in U_j$, $\delta(t)u\delta(t)^{-1}\to e$ as $t\to 0$, so by \eqref{rem:U+}, $\pi(uv')=\pi(v')$. Since $g_i=\delta(\tau_i)\sigma_ih_iu_i$, 
		\begin{equation}\label{eq:temp4}
		\lVert \pi(g_iv') \rVert = \lVert \delta(\tau_i)\pi(\sigma_ih_iu_iv') \rVert = \lvert \tau_i \rvert^{m(v',\delta)}\lVert \sigma_ih_i\pi(v') \rVert.
		\end{equation}
	Combining \eqref{eq:temp2}, \eqref{eq:temp3} and \eqref{eq:temp4}, since $m(v',\delta)\beta=(\delta,\delta)$, we get $\lVert \sigma_ih_i\pi(v') \rVert\to 0$. Since $\sigma_i\in S_{\delta}$ and $h_i\in H_j$, we conclude that $\pi(v')$ is $S_\delta H_j$-unstable in $\tilde{V}$. 
		
		Thus $S_\delta H_j$ is a reductive $\bbQ$-group acting on $\tilde V$ over $\bbQ$ and $\pi(v')\in \tilde V(\bbQ)\setminus\{0\}$ is an unstable vector for this action. Therefore by (a) of Kempf's theorem, there exists $\lambda\in X_\ast(S_\delta H_j)$ such that $\lambda(t)\pi(v')\to 0$ as $t\to 0$. 
		
		Since $S_\delta H_j\cap S$ is a maximal $\bbQ$-split torus of $S_\delta H_j$, by the conjugacy of maximal $\bbQ$-split tori,  there exists $l\in (S_\delta H_j)(\bbQ)$ such that $\delta_l:=l \lambda l^{-1} \in X_\ast(S_\delta H_j\cap S)$. So $\delta_l(t)(l\pi(v'))\to 0$ as $t\to 0$. Now $l\pi(v')=\pi(lv')$. So for any $\chi\in\charS$, 
		\begin{equation*} 
		\text{if $(\pi(lv'))_\chi\neq 0$, then } \langle \chi, \delta_l\rangle >0.
		\end{equation*}
		
		Since $S_\delta H_j\cap S=\ker\delta^\vee$, we have $1=\delta^\vee(\delta_l(t))=t^{\langle\delta^\vee,\delta_l\rangle}$ for all $t\neq 0$, and hence
		\begin{equation} \label{eq:perp}
		(\delta, \delta_l)=\langle\delta^\vee,\delta_l\rangle=0. 
		\end{equation}
		
		Since $l\in P_j(\bbQ)=P_{v'}(\bbQ)$, as we noted after the statement of Kempf's theorem,
		\begin{equation*} 
	\delta\in\Lambda_{lv'} \text{ and } m(lv',\delta)=m(v',\delta).
		\end{equation*}
		
		For a positive integer $N$, let $\delta_N=N\delta + \delta_l\in \cocharS$, in the additive notation. For $N$ large enough, we claim that
		\begin{equation} \label{eq:destablize_faster}
		 \frac{m(lv',\delta_N)}{\norm{\delta_N}}>\frac{m(\delta,lv')}{\norm{\delta}}=B_{lv'},
		\end{equation} 
		which will contradict the maximality of $B_{lv'}$.
		
		For any $w\in V$, we write $w=\sum_{\chi\in \charS} w_\chi$, where $w_\chi\in V_\chi$. Note that for any nonzero $w\in V$ and $\lambda\in \cocharS$, we have \[
		m(w,\lambda)=\min\{\langle \chi,\lambda\rangle:\chi\in\cocharS,\, w_\chi\neq 0\}.
		\]
		
		So to prove \eqref{eq:destablize_faster}, we pick any $\chi\in \charS$ such that $(lv')_\chi\neq 0$, and we will show that for all sufficiently large $N$,
		\begin{equation}\label{eq:higher_speed}
		\frac{\langle \chi, \delta_{N} \rangle}{\norm{\delta_{N}} } > \frac{m(\delta,lv')}{\norm{\delta}}.
		\end{equation}
		
		By definition 
		$m(lv',\delta)\leq\langle\chi,\delta \rangle$. First suppose that $\langle\chi,\delta \rangle>m(lv',\delta)$. Then 
		\[
		\lim_{N\to\infty} \frac{\langle \chi, \delta_{N} \rangle}{\norm{\delta_{N}} }=\frac{\langle \chi, \delta \rangle}{\norm{\delta}}>\frac{m(lv',\delta)}{\norm{\delta}},
		\]
		because $\langle\chi,\delta_l \rangle<\infty$. Therefore \eqref{eq:higher_speed} follows for all sufficiently large $N$.
		
		Now suppose that $\langle\chi,\delta \rangle=m(lv',\delta)$. Since $m(lv',\delta)=m(v',\delta)$, we have $(lv')_\chi\in \tilde V$. Since $\pi$ is $S$-equivariant, we have $(\pi(lv'))_\chi=(lv')_\chi\neq 0$. Therefore $\langle \chi, \delta_l \rangle > 0$.
	
		To prove \eqref{eq:higher_speed}, we define an auxiliary function:
		\begin{equation*}
		f(s) 
		= \frac{\langle \chi, \delta + s\cdot \delta_l \rangle ^ 2}{\norm{\delta + s \cdot \delta_l} ^ 2}
		:= \frac{\langle \chi , \delta \rangle ^ 2 + 2s\langle \chi, \delta \rangle\langle \chi, \delta_l \rangle + s^2 \langle \chi, \delta_l\rangle ^2}
		{(\delta, \delta) + 2s(\delta, \delta_l) + s^2(\delta_l, \delta_l)}, \,\forall s\in\bbR.
		\end{equation*}
		Compute its derivative at $0$:
		\begin{equation*}
		f'(0) = \frac{2\langle \chi,\delta\rangle\langle\chi,\delta_l\rangle(\delta, \delta) - 2\langle\chi, \delta\rangle^2(\delta, \delta_l)}
		{(\delta, \delta)^2}.
		\end{equation*}
		We have $\langle \chi, \delta \rangle =m(lv',\delta)> 0$ and $\langle \chi, \delta_l \rangle > 0$. Also  $(\delta, \delta_l)=0$ by \eqref{eq:perp}. Therefore $f'(0) > 0$. Hence for $N$ large we have
		\begin{equation}\label{eq:temp14}
		f(1/N) > f(0).
		\end{equation}
		Now \eqref{eq:higher_speed} follows because each side of \eqref{eq:temp14} is the square of each corresponding side of \eqref{eq:higher_speed}. Therefore \eqref{eq:destablize_faster} holds, contradicting the maximality of $B_{lv'}$.
	\end{proof}
	
	\subsection{Reduction to fundamental representations}
	Let $W_1=\bbR^3$ and $W_2=\bigwedge^2\bbR^3$. Let $w_1=e_1\in W_1$ and $w_2=e_1\wedge e_2\in W_2$. Let $\omega_1$ and $\omega_2$ be the highest weights of $W_1$ and $W_2$. Then $\omega_1$ and $\omega_2$ are the fundamental weights of $G$, and any dominant integral weight is a non-negative integral linear combination of $\omega_1$ and $\omega_2$.
	
	\begin{lem}\label{lem:reduction_to_fundamental_rep}
		Let $W$ be an irreducible representation of $G$ with highest weight $\omega=n_1\omega_1+n_2\omega_2$, where $n_1,n_2$ are non-negative integers, and let $w\in W$ be a highest weight vector. Then for any real-analytic map $\psi\colon I\to G$, where $I\subset\R$ is a nontrivial compact interval, there exists a constant $c>0$ such that for any $h_1,h_2\in G$,
		\[
		\sup_{s\in I} \norm{h_1\psi(s)h_2w} \geq c\cdot\left(\min_{1\leq i\leq 2}\sup_{s\in I} \norm{h_1\psi(s)h_2w_i}\right)^{n_1+n_2}.
		\]
	\end{lem}
	
	\begin{proof}
		Let the notation be as in the beginning of this section. We have $G=KS_0U_0$. Hence for $g\in G$, we can write $g=ktu$ for $k\in K, t\in S_0=S$ and $u\in U_0$. We note that $w_1$ and $w_2$ are both fixed by $U_0$. Taking any $K$-invariant norms on $W_1$ and $W_2$, we have $\norm{gw_1}=\abs{\omega_1(t)}\norm{w_1}$ and $\norm{gw_2}=\abs{\omega_2(t)}\norm{w_2}$. We take a $K$-invariant norm on $W$ such that $\norm{w}=\norm{w_1}^{n_1}\norm{w_2}^{n_2}$. Then for any $g\in G$,
\(
\norm{gw} = \norm{gw_1}^{n_1}\norm{gw_2}^{n_2}.
\)
		Now let
\[
F(g)=\norm{h_1gh_2w}^2
\quad\mbox{and}\quad
F_i(g)=\norm{h_1gh_2w_i}^2,\quad i=1,2.
\]
Then $F, F_1, F_2$ are regular functions on $G$, and
\[
F(g)=F_1(g)^{n_1}F_2(g)^{n_2}.
\]
Let $Z$ be the Zariski closure of $\psi(I)$ in $G$. Since $\psi$ is analytic, $Z$ is an irreducible algebraic set.
		We use the norm
\[
\norm{F}=\sup_{s\in I}\abs{F(\psi(s))}
\]
on the space of regular functions on $Z$.
We claim that for any positive integers $d_1$ and $d_2$, there exists a constant $c=c(d_1,d_2)>0$ such that for any polynomials $E_1$ and $E_2$ of degrees $d_1$ and $d_2$ respectively on $Z$, we have $\norm{E_1E_2}\geq c\norm{E_1}\norm{E_2}$. Indeed, by homogeneity we only need to check this for $\norm{E_1}=\norm{E_2}=1$, and then the possible values of $\norm{E_1E_2}$ form a compact subset of $\bbR_{>0}$. Therefore,
		\[
		\norm{F}\geq c\norm{F_1}^{n_1}\norm{F_2}^{n_2}\geq c\cdot\left(\min_{1\leq i\leq 2} \norm{F_i}\right)^{n_1+n_2}.
		\]
	\end{proof}

\subsection{From fundamental representations to the standard representation}
Let the notation be as before: We fix $(a,b)\in\bbR^2$, $I=[s_0,s_1]\subset\bbR$ for some $s_0<s_1$, and for any $s\in I$ and $t\in\bbR$, we have 
	$$\phi_{a,b}(s)= {\begin{pmatrix}
		1 & s & as+b \\
		& 1 & \\
		& & 1
		\end{pmatrix}} \text{ and } g_t=\begin{pmatrix}
		e^{2t} \\
		& e^{-t}  \\
		& & e^{-t}
		\end{pmatrix}.$$

The following observation allows us to reduce our possibilities from all fundamental representations to only the standard representation. 

		\begin{lem}\label{lem:exterior_square}
		There exists a constant $C_I$ depending only on $I$ such that for any non-zero $v\in W_2(\Z)=\bigwedge^2\bbZ^3$ and $t\geq 0$,
		\begin{equation} \label{eq:ext2}
	\sup_{s\in I} \norm{g_{t}\phi_{a,b}(s)v}\geq C_Ie^t.
    \end{equation}
	\end{lem}
	
	\begin{proof}
	    Let $e_1,e_2,e_3$ denote the standard basis of $\bbR^3$ (and $\bbZ^3$).
		We write $e_{ij}=e_i\wedge e_j$ for the standard basis of $\bigwedge^2\bbR^3$.
		For any $s\in I$ and $t\geq 0$, one can readily compute the matrix of $g_t\phi_{a,b}(s)$ in the standard basis $(e_{23},e_{13},e_{12})$:
\[{\textstyle\bigwedge}^2g_t\phi_{a,b}(s) = \begin{pmatrix}
				e^{-2t}		& 	0	 &	 0 \\
				   se^t		 &	 e^t 	 &	0 \\
				-e^{t}(as+b) 	&	0     	&	e^{t} 
			\end{pmatrix}.
\]
		So, for any $v=\begin{pmatrix}	p \\
						q \\
						r
			\end{pmatrix}$ in $\bigwedge^2\bbR^3$, $s\in I$ and $t\geq 0$, 
\[
g_t\phi_{a,b}(s)v = \begin{pmatrix}
					e^{-2t}p		\\
					e^t(sp + q)	\\
				e^{t}[-(as+b)p + r]
			\end{pmatrix}.
\]
	Now observe that
	\[
	\max\{\abs{s_0p+q},\abs{s_1p+q}\}\geq \min\left\{\frac{\abs{p}(s_1-s_0)}{2},\frac{\abs{q}(s_1-s_0)}{\abs{s_0}+\abs{s_1}}\right\},
	\]
	so that if $(p,q,r)\in\bbZ^3\setminus\{0\}$, then
\[
		\max_{s\in \{s_0,s_1\}} \{\abs{sp+q},\abs{-(as+b)p+r} \}\geq C_I,
\]
where
\[
		C_I= \min\left\{\frac{s_1-s_0}{2},\frac{s_1-s_0}{\abs{s_0}+\abs{s_1}},  1\right\}>0,
\]
		because if $(p,q)=(0,0)$ then $\abs{-(as+b)p+r}=\abs{r}\geq 1$. 
		So \eqref{eq:ext2} follows.
	\end{proof}

    By combining the above results we obtain the following: 
	
	\begin{prop} \label{thm:notclosed}  Let $V$ be a finite-dimensional representation of $G$ defined over $\bbQ$ and let $v_0\in V(\bbQ)\setminus \{0\}$. Suppose that $Gv_0$ is not Zariski closed. Then given $C>0$ there exists $R>0$ such that the following holds: There exists $t_0>0$ such that for any $t>t_0$  and any $\gamma\in\Gamma$, if
	\begin{equation}\label{eq:bnd-C}
	\sup_{s\in I} \norm{g_t\phi_{a,b}(s)\gamma v_0}\leq C,
	\end{equation}
	then there exists $v\in \bbZ^3\setminus\{0\}$ such that
	\begin{equation} \label{eq:bnd-R}
	\sup_{s\in I} \norm{g_t\phi_{a,b}(s)v}\leq R.
	\end{equation}
	\end{prop}
	
	\begin{proof} 
Let $S$ denote the boundary of $Gv_0$. By \cite[Lemma 1.1]{Kem78}, there exists a representation $V'$ of $G$ and a $G$-equivariant polynomial map $f\colon V\to V'$, both defined over $\bbQ$, such that $S=f^{-1}(0)$. It follows that $f(v_0)\in V'(\bbQ)$ is unstable in $V'$, and there exists a constant $C'>0$ and a norm on $V'$ such that \eqref{eq:bnd-C} holds for $(f(v_0), V', C')$ in place of $(v_0, V, C)$. Hence by replacing $v_0$ with $f(v_0)$ we may assume that $v_0$ is unstable in $V$.
		
		Now we can apply \Cref{prop:reduction_to_eigenvector}, and conclude that there exists an irreducible representation $W$ of $G$ defined over $\bbQ$, a highest weight vector $w\in W(\bbQ)$, an element $g_0\in G(\bbQ)$, and a constant $D>0$ such that for any $t\geq 0$ and $\gamma\in\Gamma$ if \eqref{eq:bnd-C} holds, then 
		\begin{equation*}
		\sup_{s\in I}\norm{g_{t}\phi_{a,b}(s)\gamma g_0 w}\leq D.
		\end{equation*}
		Combined with \Cref{lem:reduction_to_fundamental_rep}, this implies that there exists $D'>0$, such that for any $t\geq 0$ and $\gamma\in \Gamma$, if \eqref{eq:bnd-C} holds, then 
		\begin{equation} \label{eq:bnd-Dprime}
		\min_{1\leq j\leq 2}\sup_{s\in I} \norm{g_{t}\phi_{a,b}(s)\gamma g_0 w_j}\leq D'.
		\end{equation}
		Since $g_0\in G(\bbQ)$, there exists $N\in\bbN$ such that $N\cdot\Gamma g_0w_1\subset W_1(\bbZ)=\bbZ^3$ and $N\cdot\Gamma g_0w_2\subset W_2(\bbZ)=\bigwedge^2\bbZ^3$.
		
		By \Cref{lem:exterior_square}, for any $t\geq 0$ and $\gamma\in\Gamma$, since $v_2:=N\gamma g_0w_2\in \bigwedge^2\bbZ^3\setminus\{0\}$, we have  $\sup_{s\in I} \norm{g_t\phi_{a,b}(s)v_2}\geq C_Ie^t$. Set $R=ND'$ and $t_0:=\log{RC_I^{-1}}$. Then for any $t>t_0$ and $\gamma\in\Gamma$, we have
		\[
		\sup_{s\in I} \norm{g_{t}\phi_{a,b}(s)(\gamma g_0w_2)}> R/N=D';
		\]
	    and hence if \eqref{eq:bnd-Dprime} holds, then $\sup_{s\in I} \norm{g_{t}\phi_{a,b}(s)(\gamma g_0w_1)}\leq D'$. 
	    
	    Therefore, for any $t\geq t_0$ and $\gamma\in\Gamma$, if \eqref{eq:bnd-C} holds, then  \eqref{eq:bnd-Dprime} holds, so the non-zero vector $v=N\gamma g_0w_1\in\bbZ^3$ satisfies \eqref{eq:bnd-R}, as desired.
		\end{proof}
	
	\section{Dynamics in the standard representation and~Diophantine~conditions}
	In this section we relate asymptotic dynamics of the $g_t$-action on the curves $\{\phi_{a,b}(s)v:s\in I\}$ for nonzero $v\in\bbZ^3$ in the standard representation, with some Diophantine approximation properties of the vector $(a,b)$. 
Our first lemma characterizes the condition $(a,b)\in\Wtwo$ in terms of vectors of bounded size in the lattices $g_t\phi_{a,b}(s)\Z^3$, $s\in I$.

	\begin{lem}\label{lem:interpretation_of_Wtwo}
		The following 
		are equivalent:
		\begin{enumerate}[label=(\arabic*)]
		    \item[\rm (1)] 
		    $(a,b)\in\Wtwo$.
		\item[\rm (2)] 
		There exist $t_i\to\infty$, $\{v_i\}\subset \bbZ^3\setminus\{0\}$ and $R>0$ such that for all $i$,
		\begin{equation} \label{eq:temp6}
		\sup_{s\in I}\norm{g_{t_i}\phi_{a,b}(s)v_i}\leq R.
		\end{equation}

		\end{enumerate} 
	\end{lem}
	
	\begin{proof}
		We write $v_i=\begin{pmatrix}
		p_{1,i}\\p_{2,i}\\q_i
		\end{pmatrix}\in \Z^3\setminus\{0\}$.
It follows that
	\begin{equation} \label{eq:temp6a}
		g_{t_i}\phi_{a,b}(s)v_i=\begin{pmatrix}
		e^{2t_i}\bigl((bq_i+p_{1,i})+(aq_i+p_{2,i})s\bigr)\\
		e^{-t_i}p_{2,i}\\
		e^{-t_i}q_i
		\end{pmatrix}.
		\end{equation}
		
		(2)$\Rightarrow$(1): Let $s_0<s_1$ such that $I=[s_0,s_1]$ and \eqref{eq:temp6} holds for some $R>0$. Let 
		\(
		R_1=\norm{\begin{pmatrix}1& s_0 \\ 1& s_1\end{pmatrix}^{-1}}\cdot R
	    \),
		where $\norm{\cdot}$ denotes the operator norm with respect to the sup-norm. Then for all $i$, the following system of inequalities hold:
		\begin{equation} \label{eq:temp6b}
		\begin{cases}
		\abs{q_ib+p_{1,i}}\leq R_1e^{-2t_i}, \\
		\abs{q_ia+p_{2,i}}\leq R_1e^{-2t_i}, \\
		\abs{q_i}\leq Re^{t_i}.
		\end{cases}
		\end{equation}
		
		\subsubsection*{Case 1} Suppose a subsequence of $\{q_i\}$ is bounded.
		
		After passing to a subsequence, we may assume that $q_i=q$ is a constant. Since $qa$ and $qb$ are fixed, $\bbZ$ is discrete and $R_1e^{-2t_i}\to 0$, the first two equations from \eqref{eq:temp6b} force that $qb+p_{1,i}=0$ and $qa+p_{2,i}=0$ for all large $i$. Since $(p_{1,i},p_{2,i},q_i)\neq 0$, we conclude that $q\neq 0$ and  $(a,b)\in \bbQ^2$.

		\subsubsection*{Case 2} Suppose $\abs{q_i}\to\infty$ as $i\to\infty$.
		
		Put $R_2=R_1R^2$. Then \eqref{eq:temp6b} shows that for all $i$,
		\begin{equation} \label{eq:temp6c}
		\begin{cases}
		\abs{q_ib+p_{1,i}}\leq R_2\abs{q_i}^{-2} \\
		\abs{q_ia+p_{2,i}}\leq R_2\abs{q_i}^{-2}.
		\end{cases}
		\end{equation}
		Therefore 
		$(a,b)\in\Wtwo$. Combining both cases we proved that (2)$\Rightarrow$(1).

		(1)$\Rightarrow$(2): Suppose $(a,b)\in \Wtwo$. By \eqref{eq:Wtwo} we pick a sequence $(p_{1,i},p_{2,i},q_i)\in \bbZ^3$ such that $0\neq q_i\to\infty$ and \eqref{eq:temp6c} holds for some $R_2\geq 0$. Then for $t_i=\log \abs{q_i}$, we get \eqref{eq:temp6b} for $R_1=R_2$ and $R=1$ and $\abs{p_{2,i}}\leq R_1e^{-2t_i}+\abs{a}e^{t_i}$. So in view of \eqref{eq:temp6a}, we get that \eqref{eq:temp6} holds for $R=(\abs{s_0}+\abs{s_1}+1)R_1+\abs{a}+1$, where $I=[s_0,s_1]$. This completes the proof of (1)$\Rightarrow$(2).
	\end{proof}

The second lemma shows that $(a,b)\in\Wtwooo$ if and only if the curve $\phi_{a,b}$ is entirely sent to the cusp under the action of $g_t$ in the space of lattices, along a subsequence of times $t_i$ going to infinity.
	
	\begin{lem}\label{lem:interpretation_of_Wtwooo}
		We have $(a,b)\in\Wtwooo$ if and only if there exist $t_i\to\infty$ and $\{v_i\}\subset \bbZ^3\setminus\{0\}$ such that 
		\begin{equation*}
		\sup_{s\in I}\norm{g_{t_i}\phi_{a,b}(s)v_i}\to 0.
		\end{equation*} 
	\end{lem}
	\begin{proof}
		The proof is identical to that of \Cref{lem:interpretation_of_Wtwo}, and we leave it to the reader. 
	\end{proof}

\subsubsection{Remark} \label{subsec:badly} 
In view of the identification between $X$ and the space of unimodular lattices in $\bbR^3$, given a compact set $K\subset X$, there exists $\delta>0$ such that for any $g\in G$, if $g\bbZ^3=gx_0\in K$, then $\norm{gv}\geq \delta$ for every $v\in\bbZ^3\setminus\{0\}$. 

Suppose $(a,b)\in\Wtwooo$. By \Cref{lem:interpretation_of_Wtwooo}, there exists sequences $t_i\to\infty$ and $v_i\in\bbZ^3\setminus\{0\}$ and $i_0\in\bbN$ such that $\sup_{s\in I} \norm{g_{t_i}\phi_{a,b}(s)v_i}<\delta$ for all $i\geq i_0$. 
Therefore $g_{t_i}\phi_{a,b}(s)x_0\notin K$ for all $s\in I$, and hence $g_{t_i}\lambda_{a,b}(K)=0$ for all $i\geq i_0$. So we say that the measures $g_{t_i}\lambda_{a,b}$ escape to infinity as $i\to\infty$. 

In particular, $\{g_t\phi_{a,b}(s)x_0:t\geq 0\}$ is unbounded in $X$ for every $s\in\bbR$. Hence by Dani correspondence \cite{Dan85}, $(s,as+b)$ is not badly approximable for any $s\in\bbR$.

\bigskip
The stronger condition that the measures $g_t\lambda_{a,b}$ go to the cusp for \emph{all large\/} $t$ can only be satisfied if $(a,b)\in\bbQ^2$; this is the content of the following lemma.

	\begin{lem} \label{lem:inter_of_Q}
	The following statements are equivalent:
	\begin{enumerate}[label=(\arabic*)]
	    \item\label{itm:q} $(a,b)\in \bbQ^2$.
	    \item\label{itm:contract-v}There exists $v\in\bbZ^3\setminus\{0\}$ such that $\sup_{s\in I} \norm{g_t\phi_{a,b}(s)v}\to 0$ as $t\to\infty$. 
	    \item\label{itm:unif-bound}
	    There exists $R>0$ such that for each large $t>0$, there exists $v\in\bbZ^3\setminus\{0\}$ such that
	\begin{equation} \label{eq:uniform}
		\sup_{s\in I}\norm{g_{t}\phi_{a,b}(s)v}\leq R.
		\end{equation}
    \end{enumerate}
    \end{lem}

\begin{proof}
\ref{itm:q}$\Rightarrow$\ref{itm:contract-v}: Suppose that $a=p_1/q$ and $b=p_2/q$ for some $p_1,p_2\in\bbZ$ and $q\in\bbN$. Let $v=\begin{pmatrix}-p_2\\ -p_1 \\ q\end{pmatrix}\in\bbZ^3\setminus\{0\}$. Then for any $s\in\bbR$ and $t\geq 0$,
		\begin{equation*} 
		g_{t}\phi_{a,b}(s)v=g_t\begin{pmatrix}1 & s & as+b\\ & 1 \\&&1\end{pmatrix}\begin{pmatrix}-p_2\\ -p_1 \\ q\end{pmatrix}
		=g_{t}\begin{pmatrix} -p_2-sp_1+(as+b)q \\ -p_2\\ q \end{pmatrix}
		=e^{-t}\begin{pmatrix} 0 \\ -p_2\\ q \end{pmatrix}. 
		\end{equation*}
		Therefore \ref{itm:contract-v} holds. 
		
		\ref{itm:contract-v}$\Rightarrow$\ref{itm:unif-bound}: This is is obvious. 
		
\ref{itm:unif-bound}$\Rightarrow$\ref{itm:q}: We observe using \eqref{eq:temp6b} that \ref{itm:unif-bound} implies the following: for any $c>0$, and all sufficiently large enough $T>0$, setting $t=\log T-\log R_1$, there exists  $(p_{1},p_{2},q)\in\bbZ^3\setminus\{0\}$ such that, all the following inequalities hold:
		\begin{equation*} 
		\begin{cases} \label{eq:eqtemp6b2}
		\abs{q}\leq R_1e^{t}=T,\\
		\abs{qb+p_{1}}\leq R_1e^{-2t}\leq cT^{-1}, \\
		\abs{qa+p_{2}}\leq R_1e^{-2t}\leq cT^{-1}. 
		\end{cases}
		\end{equation*}
This implies that $a$ and $b$ are both singular real numbers.
But singular real numbers are rational, see \cite{Khintchine1926} or \cite[Remark before Theorem~XIV]{Cassels57}. 
\end{proof}

\subsubsection{Remark} \label{rem:Q-diverge}  Suppose $(a,b)\in\bbQ^2$. 
Given a compact set $K\subset X$, let $\delta>0$ be as in Remark~\ref{subsec:badly}. By the  proof of \ref{itm:q}$\Rightarrow$\ref{itm:contract-v} in \Cref{lem:inter_of_Q}, there exists $v\in\bbZ^3\setminus\{0\}$ and $t_0>0$ such that $\norm{g_{t}\phi_{a,b}(s)v}<\delta$ for all $t\geq t_0$ and all $s\in\bbR$. Therefore $g_t\phi_{a,b}(s)x_0\not\in K$ for all $t\geq t_0$ and for all $s\in\bbR$. Hence $g_t\lambda_{a,b}(K)=0$ for all $t\geq t_0$. Therefore $g_t\lambda_{a,b}$ escapes to infinity as $t\to\infty$.

\bigskip

Finally, we have a version of Lemma~\ref{lem:interpretation_of_Wtwo} for the behavior on average of the measures $g_t\lambda_{a,b}$; this will relate to the set $\Wtwoo$.
For $R>0$, define
	\begin{equation} \label{eq:defOfIR}
	\cI_{R}= \left\{ t\in[0,+\infty) :\ \exists\, v\in\bbZ^3\setminus\{0\}\ \mbox{such that}\ \sup_{s\in I}\norm{g_t\phi_{a,b}(s)v} < R\right\}.
	\end{equation}
		
	\begin{lem}\label{lem:density} The following are equivalent:
	\begin{enumerate}
		\item[\rm (1)] For every $R>0$, 
		\begin{equation} \label{eq:positive_upper_density}
		\limsup_{T\to\infty}\frac{\abs{\cI_R\cap[0,T]}}{T}>0.
		\end{equation}
		\item[\rm (2)]There exists $R>0$ such that \eqref{eq:positive_upper_density} holds. 
		\item[\rm (3)] $(a,b)\in\Wtwoo$.
	\end{enumerate}
	\end{lem}
	
We defer the proof of this result to \Cref{sec:averages}. 

\section{Reduction of linear dynamics to the standard representation}
	
	The following is one of the main technical results in this article.
It shows that \Cref{thm:notclosed} still holds even if the orbit $G\cdot v_0$ is closed, as long as it is not reduced to $\{v_0\}$.

	\begin{thm}
	\label{prop:consequence_of_linear_focusing}
	Let $V$ be a finite-dimensional representation of $G$ over $\bbQ$ and $v_0\in V(\bbQ)\setminus\{0\}$ such that $v_0$ is not $G$-fixed. Then given $C>0$ there exists $R>0$ and $t_0>0$ such that the following holds: For every $t>t_0$, if there exists $\gamma\in\Gamma$ such that
	\begin{equation*}\label{eq:temp7-1}
		\sup_{s\in I}\norm{g_{t}\phi_{a,b}(s)\gamma v_0}\leq C,
		\end{equation*}
	then there exists $v\in \bbZ^3\setminus\{0\}$ such that 
	\begin{equation*}\label{eq:boundedInStdRep-1}
		\sup_{s\in I}\norm{g_{t}\phi_{a,b}(s)v}\leq R.
		\end{equation*}
	\end{thm}
	
For the proof of the above theorem we introduce some notation and make some observations.
	
	\subsection{Linear dynamics of $g_t$ and \texorpdfstring{$\phi_{a,b}(s)$}{ϕ_(a,b)(s)} actions}
	\label{subsec:sl2}
	
	Let $V$ be an irreducible real representation of $G=\SL_3(\bbR)$
	over $\bbQ$. 
	Since $G$ is $\bbQ$-split, 
	$V\otimes\bbC$ is $G$-irreducible over $\bbC$.
	
	We express  
		\[
		g_t=c_tb_t\text{, where } b_t=\diag(e^{t/2}, e^{t/2}, e^{-t}) \text { and } c_t=\diag(e^{3t/2},e^{-3t/2},1).
		\]
		Let 
		\begin{equation} \label{eq:u12gab}
		u_{23}(s)={\begin{pmatrix}
		1 & & \\
		& 1 & s \\
		& & 1
		\end{pmatrix}}, \;
		u_{12}(s)={\begin{pmatrix}
		1 & s & \\
		& 1 & \\
		& & 1
		\end{pmatrix}}, \;
		\gab ={ \begin{pmatrix}
		1 & & b\\
		& 1 & a\\
		& & 1
		\end{pmatrix}}.
		\end{equation}
		Then we have $\phi_{a,b}(s)=u_{23}(-a)u_{12}(s)\gab $.
		Since $g_t$ commutes with $u_{23}(-a)$ and $b_t$ commutes with $u_{12}(s)$,
		\begin{equation} \label{eq:imp-gab}
		g_{t}\phi_{a,b}(s)=u_{23}(-a)g_{t}u_{12}(s)\gab \text{ and } g_{t}u_{12}(s)\gab =c_{t}u_{12}(s)b_{t}\gab .
		\end{equation}
		
		Let $H=H_2=
		\begin{pmatrix} \SL_2(\R)\\&1\end{pmatrix}<G$, and consider $V$ as the restricted representation of $H$. We have a decomposition
		$
		V=V_{1} \oplus V_{2},
		$
		where $V_1=\{v\in V\ |\ \forall h\in H,\ hv=v\}$ is the subspace fixed by $H$, and $V_2$ is the complement of $V_1$ stable under the action of $H$. Let $\pi_1$ and $\pi_2$ be the $H$-equivariant projections from $V$ to $V_1$ and $V_2$ respectively.  Since $b_t$ centralizes $H$, $\pi_1$ and $\pi_2$ are also $b_t$-equivariant. 
		
		The following observation is based on \cite[Lemma 2.3]{Sha09Duke1}.
		
		\begin{lem}[Linear dynamics of an $\SL_2$ action] 
		\label{lem:sl2}
		For any $m\geq 0$, let $W$ denote the $(m+1)$-dimensional irreducible representation of $\SL_2(\R)$. Let 
		\[ 
		a_t={\begin{pmatrix}e^t\\&e^{-t}\end{pmatrix}} \text{ and }
		u(s)={\begin{pmatrix} 1 & s \\ & 1\end{pmatrix}}.
		\]
		Then there exists $C_I>0$ such that for any $w\in W$ and $t\in\R$,
		\[
		\sup_{s\in I} \norm{a_tu(s)w}\geq e^{mt}(C_Im)^{-m}\norm{w}.
		\]
		\end{lem}
		
		\begin{proof}
		Let $e_0,\ldots,e_m$ denote a basis of $W$ such that $a_te_k=e^{(m-2k)t}e_k$ for all $k$. For any $w\in W$, we express $w=\sum_{k=0}^m w_ke_k$, where $w_k\in\R$.
	    Then
		\[
		(u(s)w)_0=\sum_{k=0}^m w_ks^k.
		\]
		Let $\norm{w}=\max_{0\leq k\leq m}\abs{w_k}$.
		Recall that $I=[s_0,s_1]$ and let $\tau_j=s_0+(j/m)(s_1-s_0)$ for $0\leq j\leq m$. Then
		\begin{equation} \label{eq:basic-SL2}
		\sup_{s\in I} \abs{(u(s)w)_0}\geq \max_{0\leq j\leq m}\abs{\sum_{k=0}^m w_k\tau_j^k}\geq  (C_Im)^{-m}\norm{w},
		\end{equation} 
		where $C_I=(1+\max\{\abs{s_0},\abs{s_1}\})/(s_1-s_0)$, 
		from an estimate for the norm of the inverse of the $(m+1)\times (m+1)$-Vandermonde matrix $(\tau_j^k)$ \cite[Theorem~1]{Gautschi62}.
		\end{proof}
		
		\begin{cor} \label{cor:ctu12}
		There exist constants $C_2>0$ and $\beta\geq 3/2$ such that for all $v\in V_2$ and all $t\geq0$,
		\begin{equation} \label{eq:H-ct}
		\sup_{s\in I}\norm{c_tu_{12}(s)v}\geq C_2e^{\beta t}\norm{v}.
		\end{equation}
		\end{cor}
		
		\begin{proof}
		Consider the action of $H\cong\SL_2(\R)$ on any irreducible component $W$ of $V_2$.
		By definition of $V_2$, the representation $W$ is non-trivial, i.e. $\dim W=m+1$, with $m\geq 1$.
		Under the identification of $H$ with $\SL_2(\R)$, we have $c_t=a_{3t/2}$ and $u_{12}(s)=u(s)$. Therefore \Cref{lem:sl2} shows that 
		\[
		\sup_{s\in I}\norm{c_tu_{12}(s)w}\geq C_2e^{\frac{3m}{2} t}\norm{w},
		\]
		where $C_2=(C_Im)^{-m}>0$. Since this holds for every $H$-irreducible component $W\subset V_2$, we indeed obtain the desired inequality for all $v$ in $V_2$, with
		\[
		\beta=\min\left\{\frac{3m}{2}\ ;\ m=\dim W-1,\ W\subset V_2\ \mbox{irreducible}\right\}.
		\]
	    \end{proof}
	
	Let $U_{12}=\{u_{12}(s)\}_{s\in\bbR}$, and $V^{U_{12}}$ be the subspace of $U_{12}$-fixed vectors in $V$. Let $\pi_{U_{12}}:V\to V^{U_{12}}$ denote the $\{c_t\}_{t\in\bbR}$-equivariant projection. We apply  \eqref{eq:basic-SL2} of \Cref{lem:sl2} to each $H$-irreducible component of $V$ to obtain the following:
	
		\begin{cor} \label{cor:U12}
	     There exists $C_4=C_4(I,V)>0$ such that for any $v\in V$, 
		\begin{equation*} \label{eq:U_12}
		\sup_{s\in I} \norm{\pi_{U_{12}}(u_{12}(s)v)}\geq C_4\norm{v}.
		\end{equation*}
		\end{cor}

		\subsection{Consequences of description of irreducible representations of $\SL(3,\R)$} Let $\omega=n_1\omega_1+n_2\omega_2$ be the highest weight of $V$, where $n_1,n_2$ are non-negative integers.
		
		\begin{lem} \label{lem:dimV1} $\dim(V_1)=1$ and its weight is $-(n_1-n_2)\omega_2$. In particular, $b_t$ acts on $V_1$ as the scalar multiplication by $e^{-(n_1-n_2)t}$.
		\end{lem}
		\begin{proof} 
		We consider the weight diagram and multiplicities of the weights of an irreducible $\SL_3$-representation as in Figure~\ref{fig:weightDiag1}; see \cite[\S13.2]{FultonHarris}. 
		
		\begin{figure}[htb]
		\centering
		\includegraphics[width=0.5\textwidth]{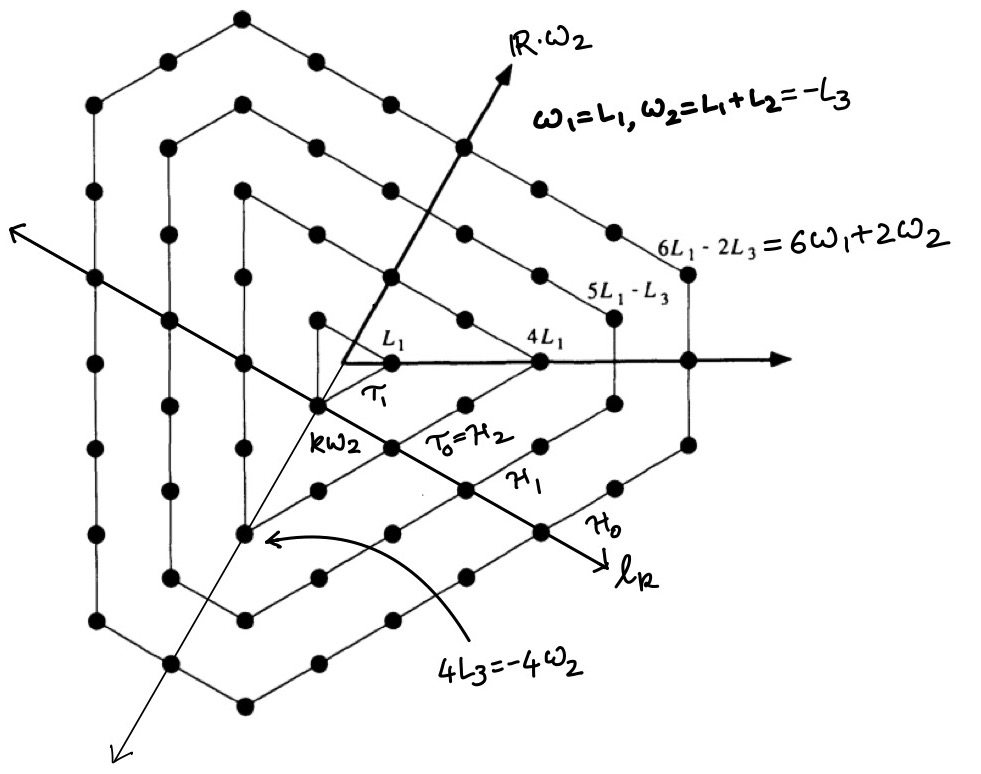}
		\caption{$\SL_3$-representation with the highest weight $6\omega_1+2\omega_2$. (Based on \cite[Figure (13.6)]{FultonHarris}.)}
		\label{fig:weightDiag1}
\end{figure}
		
				The weights of $V$ lie on hexagons $\cH_0,\ldots,\cH_{m-1}$, where $m=\min(n_1,n_2)$, and triangles $\cT_0,\ldots, \cT_{\lfloor \abs{n_1-n_2}/{3}\rfloor}$. We set $\cH_{m}=\cT_0$ and also call it a hexagon, which is degenerate. The multiplicity of a weight on any $\cH_i$ is $i+1$ and on any triangle is $m+1$.
		
		Consider any weight of $V_1$. Then it is fixed by the Weyl reflection corresponding to $H$, so it must be $k\omega_2$ for some $k\in\Z$. Let $\ell_k$ be the line perpendicular to $\omega_2$ and passing through $k\omega_2$. Then $\ell_k\cap \cH_i\neq \emptyset$ if and only if $0\leq i\leq j_k-1$, where $j_k$ is the multiplicity of $k\omega_2$ in $V$, and for each such $i$, $\ell_k\cap \cH_i$ contains the highest and the lowest weights of an irreducible representation of $H$ containing $k\omega_2$.
		
		In the above description, there is exactly one case when we have a trivial $H$-representation; that is,  when $k\omega_2$ is a vertex of the triangle $\cT_0=\cH_{m}$ and $\ell_k\cap\cH_{m}=\{k\omega_2\}$. In particular, $\dim V_1=1$. The dominant vertex of $\cT_0$ is $(n_1-n_2)\omega_1$ or $(n_2-n_1)\omega_2$. A Weyl reflection sends $\omega_1$ to $-\omega_2$. So in both  cases, $-(n_1-n_2)\omega_2$ is a vertex of $\cT_0$ and $k=-(n_1-n_2)$. This proves the claim.
		\end{proof}

		\begin{lem} \label{lem:n1<n2}
		Suppose $n_1\leq n_2$. Then all the weights occurring in $V^{U^{12}}$ are non-negative for the Lie algebra element $\diag(2,-1,-1)$ corresponding to $g_t$. 
		
		In particular, by \Cref{cor:U12}, for any $v\in V$ and $t\geq 0$,	
		\begin{equation} \label{eq:n1n2-sl2}
		\sup_{s\in I} \norm{g_tu_{12}(s)v}\geq C_4\norm{v}.
		\end{equation}
		\end{lem} 
		
		\begin{proof}
		We consider the weight diagram of an irreducible $\SL_3$-representation as in Figure~\ref{fig:Weights2}; see \cite[Proposition~12.18]{FultonHarris}.
		
		\begin{figure}[htb]
		\centering
		\includegraphics[width=0.5\textwidth]{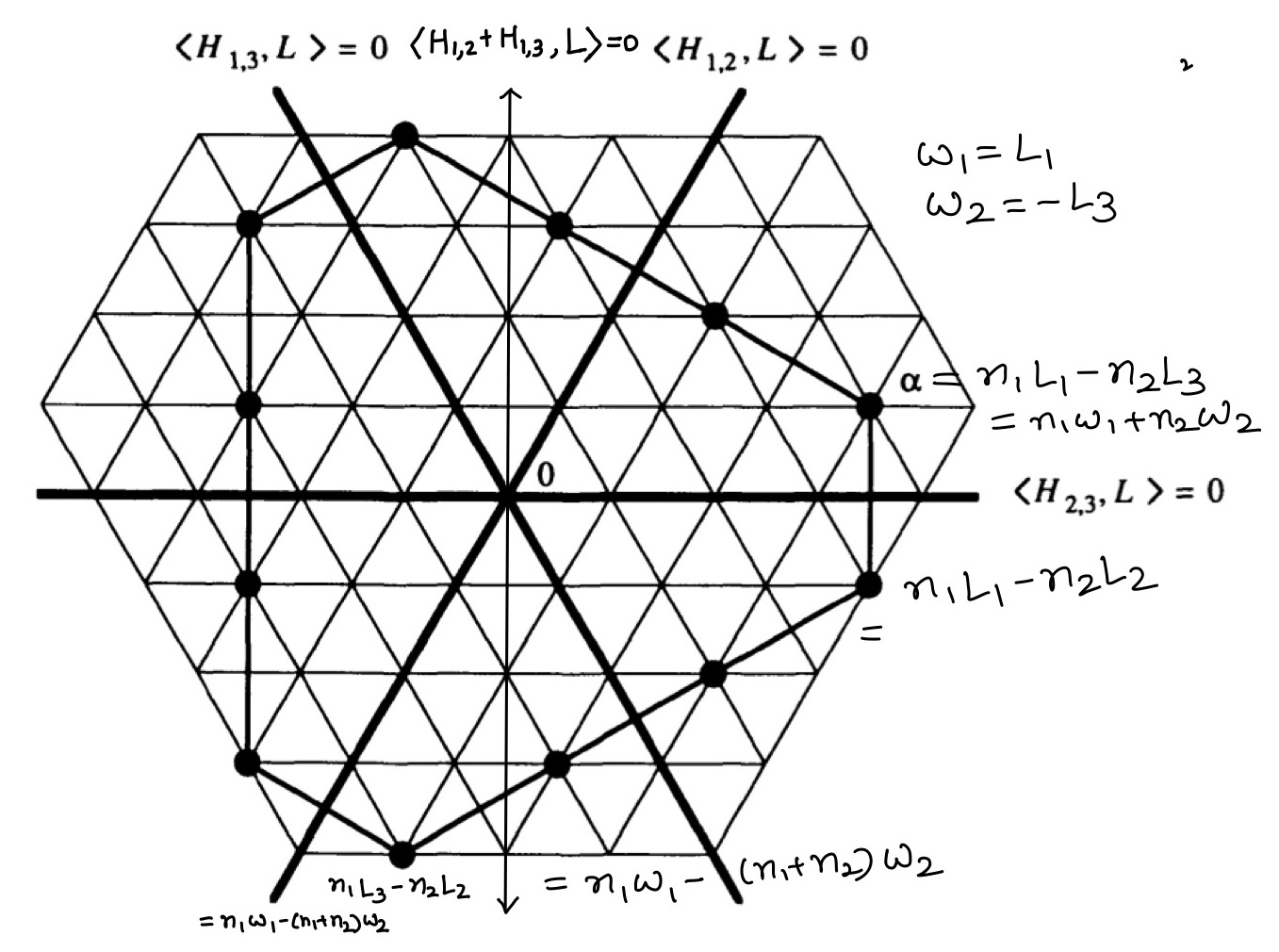}
		\caption{Arrangement of weights of an $\SL_3$-representation. (Based on \cite[Figure (12.14)]{FultonHarris}.)}
		\label{fig:Weights2}
        \end{figure}
		
		We note that the weights on the vertical line passing through the origin in the weight diagram (see~\Cref{fig:Weights2}) vanish on $\diag(2,-1,-1)$ and those on the right half take positive values. 
		
		The weights occurring in $V^{U_{12}}$ are the highest weights of irreducible representations of $H$ in $V$, and they lie on two of the sides of the hexagons $\cH_i$, where \linebreak $0\leq i\leq \min(n_1,n_2)=n_1$; if one draws the set of weights as in \Cref{fig:Weights2}, then for each $i$, one of the sides is a vertical segment from the dominant weight $(n_1-i)L_1-(n_2-i)L_3$ to the weight $(n_1-i)L_1-(n_2-i)L_2$, and the other side is the segment joining the last weight to $(n_1-i)L_3-(n_2-i)L_2$. For $\diag(2,-1,-1)$, all the weights on the vertical segment have constant non-negative value 
		\[
		2(n_1-i)+(n_2-i)\geq n_2-n_1,
		\]
		and the weight $(n_1-i)L_3-(n_2-i)L_2$ has the value 
		\[
		(n_1-i)(-1)-(n_2-i)(-1)=n_2-n_1\geq 0.
		\]
		So all the weights on both  segments have a non-negative value for $\diag(2,-1,-1)$. 
(Note that this is not the case in \Cref{fig:Weights2}, where $n_1=3>1=n_2$.)
		\end{proof}

	\subsection{Proof of \Cref{prop:consequence_of_linear_focusing}} 
		
		If $Gv_0$ is not Zariski closed, then the result follows from \Cref{thm:notclosed}. So we assume that $Gv_0$ is Zariski closed. 
		
		If the theorem fails to hold, then there exist sequences $t_i\to\infty$ and $\gamma_i\in\Gamma$ such that 
		\begin{equation}\label{eq:temp7}
		\sup_{s\in I}\norm{g_{t_i}\phi_{a,b}(s)\gamma_iv_0}\leq C,
		\end{equation}
		and for every sequence $v_i\in\bbZ^3\setminus\{0\}$, 
		\begin{equation} \label{eq:unbndd}
		    \sup_{s\in I}\norm{g_{t_i}\phi_{a,b}(s)v_i}\to\infty.
		\end{equation}
		
		Write $\norm{\cdot}_V$ to denote the operator norm on the linear space $V$.
		By \eqref{eq:temp7} and \eqref{eq:imp-gab}, for $C_1=\frac{C}{\norm{u_{23}(-a)}_{V}}>0$, for all $i$,
		\begin{align}\label{eq:temp8}
		C_1
		\geq\sup_{s\in I} \norm{g_{t_i}u_{12}(s)(\gab \gamma_iv_0)}
		=\sup_{s\in I}\norm{c_{t_i}u_{12}(s)(b_{t_i}\gab \gamma_iv_0)}.  
		\end{align}
		
		We recall that $V=V_1\oplus V_2$, where $H$ acts trivially on $V_1$, and $V_2$ is a sum of non-trivial irreducible representations of $H$, and for $j=1,2$, $\pi_j:V\to V_j$ is the corresponding projection.
		By Corollary \ref{cor:ctu12}, this implies that there exist $C_3>0$ and $\beta\geq 3/2$ such that
		\begin{equation}\label{eq:temp9}
		\forall i,\quad
		\norm{\pi_1(b_{t_i}\gab \gamma_iv_0)}\leq C_3
		\quad\mbox{and}\quad
		 \norm{\pi_2(b_{t_i}\gab \gamma_iv_0)}\leq C_3e^{-\beta t_i}.
		\end{equation}

There are two cases. We will show that each will lead to a contradiction.  
		
		\subsubsection*{Case 1:} {\em $\{\gamma_iv_0\}_{i\in\bbN}$ is unbounded in $V$.}
		
\bigskip
Let $n_1$ and $n_2$ be non-negative integers such that the highest weight of the irreducible $G$-representation $V$ is $n_1\omega_1+n_2\omega_2$.

		First suppose that $n_1>n_2$. The highest eigenvalue of $b_t$ on $V$ is $e^{(n_1/2+n_2)t}$. We pick $\eps>0$ such that $\eps(n_1/2+n_2)<\beta$. By \Cref{lem:dimV1}, $b_t$ acts on $V_1$ by the scalar $e^{-(n_1-n_2)t}$. Therefore, by \eqref{eq:temp9}, for all $i$,
		\begin{align*}
		\norm{b_{\eps t_i}\pi_1(b_{t_i}\gab \gamma_iv_0)}&\leq C_3e^{-\eps(n_1-n_2)t_i} \text{ and }\\
		\norm{b_{\eps t_i}\pi_2(b_{t_i}\gab \gamma_iv_0)}&\leq C_3e^{(-\beta+\eps(n_1/2+n_2))t_i}. 
		\end{align*}
		So $b_{(1+\eps) t_i}\gab \gamma_iv_0\to0$ as $i\to\infty$. This contradicts the fact that $Gv_0$ is Zariski closed.
		
		Hence we must have $n_1\leq n_2$. Then by \eqref{eq:n1n2-sl2} and \eqref{eq:temp8} we get $\{\gab \gamma_iv_0\}_{i\in\bbN}$ is bounded, and it follows that $\{\gamma_iv_0\}_{i\in\bbN}$ is bounded. This contradicts the assumption of Case~1. 
		
		\subsubsection*{Case 2:} {\it  $\{\gamma_iv_0\}_{i\in\bbN}$ is bounded in $V$.}
		
		\bigskip
		In this case using \eqref{eq:temp7} we will deduce that $(a,b)\in\bbQ^2$, which contradicts \eqref{eq:unbndd} by \Cref{lem:inter_of_Q} .  
		
		Since $\Gamma v_0$ a discrete subset of $V$ and $\{\gamma_iv_0\}_{i\in\bbN}$ is bounded, there exists $j\in\bbN$ such that $\gamma_iv_0=\gamma_jv_0$ for infinitely many $i\in\bbN$. Therefore, replacing $v_0$ with $\gamma_jv_0$ and $\gamma_i$ with $\gamma_i\gamma_j^{-1}\in\Gamma$ for each $i$, after passing to a subsequence we may assume that   $\gamma_i v_0=v_0$ for all $i$ and \eqref{eq:temp7} holds. Let
		\[
		F = \Stab_G(v_0).
		\]
		Since $v_0\in V(\bbQ)$ and $v_0$ is not $G$-fixed, $F$ is a proper algebraic subgroup of $G$ defined over $\bbQ$. Since $G v_0$ is Zariski closed, $G/F\cong Gv_0$ is an affine variety. So by Matsushima's criterion~\cite[\S7.10]{borel_iga}, $F$ is a reductive subgroup of $G$. Thus $F$ is a proper reductive algebraic subgroup of $G$ defined over $\bbQ$.
		
		Now \eqref{eq:temp9} implies that $\{b_{t_i}\gab v_0 \}_{i\in\bbN}$ is bounded. Therefore after passing to a subsequence we may assume that
		\[
		b_{t_i}\gab v_0\to v_\infty
		\quad\mbox{for some}\ v_\infty\in V\setminus\{0\}.
		\]
Since $Gv_0$ is Zariski closed, we may write $v_\infty = gv_0$ for some $g\in G$.
Let
\[
L=\Stab_G(v_\infty).
\]
Then $L=gFg^{-1}$ and $L\neq G$.
It is clear from \eqref{eq:temp9} that $v_\infty\in V_1$ is fixed by $H$, and we also know that $v_\infty$ is fixed by $\{b_t\}_{t\in\bbR}$ by definition of $v_\infty$. Hence $L$ is a proper reductive subgroup of $G$ which contains the group generated by $H$ and $\{b_t\}_{t\in\bbR}$; the normalizer $N_G(H)$ of $H$. But $N_G(H)$ is a maximal reductive subgroup of $G$. Hence
\[
L=N_G(H).
\]
		
		\subsubsection*{Claim.} {\it We can pick $h_0\in G(\bbQ)$ such that $v_\infty=h_0v_0$.}
		
		\bigskip
		To prove the claim, consider the center of $F$, denoted $Z(F)$. Then
		\[
		gZ(F)g^{-1}=Z(L)=Z(N_G(H))=S_2=
		\left\{\begin{pmatrix} t & 0 & 0\\
						0 & t & 0\\
						0 & 0 & t^{-2}
					\end{pmatrix}\right\}.
		\]
		Therefore, since $F$ is an algebraic subgroup of $G$ defined over $\bbQ$,  $Z(F)$ is a one-dimensional $\bbR$-split torus in $G$ defined over $\bbQ$. In particular, $Z(F)(\bbQ)$ is Zariski dense in $Z(F)$, and hence a single element, say $\gamma\in Z(F)(\bbQ)$, generates a Zariski dense subgroup of $Z(F)$. Since $g\gamma g^{-1}\in S_2$, the roots of the characteristic polynomial of $\gamma$ are $t$, $t$, and $t^{-2}$ for some $t\in\bbR\setminus\{0\}$. Since $\gamma\in\SL(3,\bbQ)$, these roots permute under the Galois action. We conclude that $t$ is fixed by this action, so $t\in\bbQ$. Hence there exists $h_0\in G(\bbQ)$ such that 
		\[
		h_0\gamma h_0^{-1}=\diag(t,t,t^{-2})\in S_2.
		\]
		Therefore $h_0 Z(F)h_0^{-1}\subset S_2$. Hence $h_0Z(F)h_0=S_2$. The centralizers of $S_2$ in $G$ is $N_G(H)=L$. Therefore the centralizer of $Z(F)$ in $G$ is conjugate to $F$ and contains $F$, so it equals $F$.  Therefore $h_0Fh_0^{-1}=L$. Since $gFg^{-1}=L$, we have $h_0g^{-1}\in N_G(L)=L$, as $L=N_G(H)$ is a maximal subgroup. Hence $h_0v_0=h_0g^{-1}v_\infty=v_\infty$. This proves the claim.
		
		\bigskip
		Thus $b_{t_i}\gab v_0=b_{t_i}\gab (h_0^{-1}v_\infty)\to v_\infty$. Since $\Stab_{G}(v_\infty)=N_G(H)$ and the orbit $Gv_\infty$ is locally compact, the map $g[N_G(H)]\mapsto gv_\infty$ from $G/N_G(H)\to V$ is a homeomorphism onto its image. Therefore 
		\[
		b_{t_i}\gab h_0^{-1}[N_G(H)]\to [N_G(H)]
		\]
		in $G/N_G(H)$ as $i\to\infty$.
		Consider the standard projective action of $G$ on $\bbP(\bbR^3)$. Then $N_G(H)$ fixes $\langle e_3\rangle$. So
		\[
		b_{t_i}\gab h_0^{-1}\langle e_3\rangle \to \langle e_3 \rangle
		\]
		as $i\to\infty$. Since $b_{t}=\diag(e^{t/2},e^{t/2},e^{-t})$, we conclude that $\gab h_0^{-1}\langle e_3\rangle=\langle e_3\rangle$. So $h_0^{-1}e_3=\lambda \gab ^{-1}e_3$ for some $\lambda\neq 0$. Since $h_0^{-1}\in G(\bbQ)$, by \eqref{eq:u12gab} we get $\lambda (-b,-a,1)\in \bbQ^3$. So $\lambda\in\bbQ$, and hence $(a,b)\in \bbQ^2$. 
	
	   As noted earlier, $(a,b)\in \bbQ^2$ contradicts \eqref{eq:unbndd} in view of by \Cref{lem:inter_of_Q}. \qed
	
	\section{The Dani-Margulis criterion for non-escape of mass} \label{sect:tightness}

	Let $\lambda_{a,b}$ be the probability measure on $X$ given by \eqref{eq:lambda_ab}.  
	The goal of this section is to give a necessary and sufficient condition for non-escape of mass for $\{g_{t_i}\lambda_{a,b}\}$ as $t_i\to\infty$.
	
Define the unipotent one-parameter subgroup
\begin{equation*}
\label{eq:Unip-W}
W=\{w(r):={ \begin{pmatrix} 
	    1 & r & ar \\ & 1 \\ && 1
	    \end{pmatrix}}: r\in\bbR\}\subset G.
\end{equation*}

For any $t\in\bbR$ and $s\in I=[s_0,s_1]$, we have
\begin{equation} \label{eq:unipotent}
	    g_t\phi_{a,b}(s)=w(r)h_t\text{, where } r=e^{3t}(s-s_0),\,  
	    h_t=g_t\phi_{a,b}(s_0).
	\end{equation}
Hence the trajectory $\{g_t\phi_{a,b}(s)x_0:s\in I\}$ equals $\{w(r)h_tx_0:r\in [0,e^{3t}\abs{I}]\}$, which is a segment of a unipotent orbit.

By a criterion due to Dani and Margulis for analyzing non-escape of mass for unipotent trajectories on the space of unimodular lattices, we obtain the following: 
	
	\begin{prop} \label{thm:KMnondivergence}
		For any $\eps > 0$ and $R > 0$, there exist a compact set $K \subset X$ and 
		$t_I\geq 0$ such that for any $t\geq t_I$, one of the following two possibilities holds:
		\begin{enumerate}
			\item[\rm (1)] 
			$\lvert \{ s\in I \colon g_t\phi_{a,b}(s)x_0 \in K \} \rvert \geq (1-\eps)\lvert I \rvert$.
			\item[\rm (2)] There exists $w \in \bbZ^3\setminus\{0\}$  such that
			\begin{equation*}
			\sup_{s\in I}\lVert g_t\phi_{a,b}(s) w\rVert < R.
			\end{equation*}	
		\end{enumerate}
	\end{prop}
	
	\begin{proof}
	     By the result of Dani and Margulis~\cite[1.1.~Theorem]{Dani+Margulis:Invent89}, given any $\eps > 0$ and $R > 0$, we can pick a compact set $K \subset X$ such that given any finite interval $I\subset \bbR$ and any $t\geq 0$ one of the following three statements holds: the above condition (1), or the above condition (2) or  the following additional condition (3): there exists $w\in \bigwedge^2\bbZ^3\setminus\{0\}$ such that 
	     \[
	     \sup_{s\in I} \norm{g_t\phi_{a,b}(s)w}\leq R.
	     \]
	     Now if (3) holds for some $t$, then $C_Ie^t\leq R$ by Lemma~\ref{lem:exterior_square}. So the additional condition (3) will not occur for $t\geq t_I:=\log(C_I^{-1}R)$.  
	\end{proof}
	
	\begin{prop} \label{prop:consequence_of_divergence}
	Let $\{t_i\}_{i\in\bbN}$ be a sequence of real numbers such that $t_i\to\infty$. The following are equivalent: 
		\begin{enumerate}
		\item[\rm (1)] For every compact set $K\subset X$, $g_{t_i}\phi_{a,b}(I)\cap K=\emptyset$ for all large $i$.
		\item[\rm (2)] There exists $\eps>0$ such that for every compact set $K\subset X$, 
		\begin{equation} \label{eq:temp15a}
		g_{t_i}\lambda_{a,b}(K)\leq 1-\eps \text{, for all large $i$.} 
		\end{equation}
		\item[\rm (3)] There exist vectors $\{v_i\}\subset \bbZ^3\setminus\{0\}$ such that
		\begin{equation}\label{eq:temp15}
		\sup_{s\in I}\norm{g_{t_{i}}\phi_{a,b}(s)v_{i}}\to 0\text{ as } i\to\infty.
		\end{equation}
		\end{enumerate}
	\end{prop}
	
	\begin{proof}   (1) $\Rightarrow$ (2) is obvious.
	
		  (2) $\Rightarrow$ (3): Fix $\eps>0$ so that \eqref{eq:temp15a} holds for every compact $K\subset 
		X$. For each $j\in\bbN$ and $R=1/j$, 
		obtain a compact set $K_j\subset X$ as in  \Cref{thm:KMnondivergence}. Then for all $i$,
		\[
		g_{t_i}\lambda_{a,b}(K_j)
		=\frac{1}{\abs{I}}\abs{\{s\in I:g_{t_i}\phi_{a,b}(s)x_0\in K_j\}}.
		\]
		So, by \eqref{eq:temp15a} there exists $i_j\in\bbN$ such that the possibility~(1) of  \Cref{thm:KMnondivergence} does not hold for all $i\geq i_j$; and hence its second assertion~(2) must hold. So for all $i\geq i_j$, there exists $w_{j,i}\in\bbZ^3\setminus\{0\}$ such that  
		\begin{equation} \label{eq:temp15b}
			\sup_{s\in I}\lVert g_{t_i}\phi_{a,b}(s)w_{j,i}\rVert \leq  1/j.
			\end{equation}
		For each $i\geq i_1$, let $j$ be maximal such that $i\geq i_j$, and put $v_i=w_{j,i}$. Then \eqref{eq:temp15} follows from \eqref{eq:temp15b}.  
		
	    (3) $\Rightarrow$ (1) is a straightforward consequence of Mahler's compactness criterion.
	\end{proof}
	
	\begin{proof}[Proof of \Cref{thm:main_thm_non_escape_of_mass}]
		To say that the sequence of $g_t$-translates of $\lambda_{a,b}$ has no escape of mass means that there exists a sequence $t_i\to\infty$ such that condition~(2), and hence equivalently condition~(3), of \Cref{prop:consequence_of_divergence} fails to hold. It remains to apply \Cref{lem:interpretation_of_Wtwooo}.
	\end{proof}

	\section{Ratner's theorem and a linear dynamical criterion \texorpdfstring{\\}{} for avoidance of singular sets}\label{sect:linearization}
The collection of probability measures on the one-point compactification, say $\bar X=G/\Gamma\cup\{\infty\}$, of $X = G/\Gamma$ is compact with respect to the weak-* topology on $\bar X$. So given any sequence $t_i\to\infty$, after passing to a subsequence, we obtain that $g_{t_i}\lambda_{a,b}$ converges to a probability measure $\bar\mu$ on $\bar X$. Let $\mu$ denote the restriction of $\bar\mu$ to $X$. Then $g_{t_i}\lambda_{a,b}$ converges to 
$\mu$ 
with respect to the weak-* topology; that is, for all $f\in C_c(X)$, we have
\[
\lim_{i\to\infty} \int_{X} f\,\mathrm{d}(g_{t_i}\lambda_{a,b})=\int_{X} f\, \mathrm{d}\mu.
\]
For the proposition below, recall that
\[
W = \left\{w(r)={ \begin{pmatrix} 
	    1 & r & ar \\ & 1 \\ && 1
	    \end{pmatrix}}:\ r\in\R\right\}.
\]

	\begin{prop}
	    \label{prop:U-inv}
	    Suppose that $\mu$ is a weak-* limit of $g_{t_i}\lambda_{a,b}$ for a sequence $t_i\to\infty$. Then $\mu$ is invariant under the action of $W$. 
	\end{prop}
	
	\begin{proof}
		By \eqref{eq:unipotent}, for any $t>0$ and any $f\in C_c(X)$,
	\begin{equation} \label{eq:deflambda}
	\int_{X} f\,\mathrm{d}(g_t\lambda_{a,b}) =\fint_I f(g_t\phi_{a,b}(s)x_0)\,\mathrm{d}s=\frac{1}{e^{3t}}\int_0^{e^{3t}} f(w(r)h_tx_0)\,\mathrm{d}r,
	\end{equation}
	where $h_t=g_t\phi_{a,b}(s_0)$. So for any $r_0\in\R$,
	\begin{equation} \label{eq:wr-inv}
	\begin{array}{l}
	    \int_{X} f(w(r_0)x)\,\mathrm{d}(g_t\lambda_{a,b})(x)=\frac{1}{e^{3t}}\int_0^{e^{3t}} f(w(r_0)w(r)h_tx_0)\,\mathrm{d}r\\
	    =\frac{1}{e^{3t}}\int_{r_0}^{r_0+e^{3t}} f(w(r)h_tx_0)\,\mathrm{d}r
	    =\frac{1}{e^{3t}}\int_0^{e^{3t}} f(w(r)h_tx_0)\,\mathrm{d}r +\eps_t\\
	    =\int_{X} f\,\mathrm{d}(g_t\lambda_{a,b})+\eps_t,
	    \end{array}
	\end{equation}
	where $\abs{\eps_t}\leq 2r_0e^{-3t}\norm{f}_\infty$. Then observe that the left-most (resp., right-most) term of \eqref{eq:wr-inv} converges to $\int_{X} f(w(r_0)x)\,\mathrm{d}\mu(x)$ (resp.,  to $\int_{X} f\,\mathrm{d}\mu$) as $t=t_i\to\infty$.
	\end{proof}
	
	\begin{prop} \label{prop:Ave-W-inv}
	    Suppose that $\mu$ is a weak-* limit of $(1/T_i)\int_{0}^{T_i} g_t\lambda_{a,b}\,\mathrm{d}t$ for a sequence $T_i\to\infty$. Then $\mu$ is invariant under the action of $W$.
	\end{prop}
	
	\begin{proof} We perform the average over $t\in[0,T_i]$ in \eqref{eq:deflambda} and \eqref{eq:wr-inv} and take the limit as $i\to\infty$ to conclude that $\mu$ is $w(r_0)$-invariant. 
	\end{proof}

	With \Cref{prop:U-inv}, we will be able to apply Ratner's description of ergodic invariant measures for actions of unipotent one-parameter subgroups on $X$ to analyze the limiting distributions of $\{g_t\lambda_{a,b}\}$ as $t\to\infty$. For this purpose we will apply what is now called `the linearization technique'~\cite{DM93}. 
	
	Let $\pi \colon G \rightarrow X$ denote the natural quotient map. Let $\mathcal{H}$ denote the collection of closed connected subgroups $H$ of $G$ such that $H\cap\Gamma$ is a lattice in $H$, and such that a unipotent one-parameter subgroup contained in $H$ acts ergodically with respect to the $H$-invariant probability measure on $H/H\cap\Gamma$. Then any $H\in\cH$ is a real algebraic group defined over $\bbQ$ \cite[(3.2)~Proposition]{Shah:MathAnn91}. In particular, $\mathcal{H}$ is a countable collection~\cite{Ra91}.
	
	Let $W$ be a one-parameter unipotent subgroup of $G$. For a closed connected subgroup $H$ of $G$, define
	\begin{equation*}
	N(H,W) = \{g\in G\colon g^{-1}Wg \subset H \}.
	\end{equation*}
	Now, suppose that $H\in\mathcal{H}$. We define the associated singular set
	\begin{equation*}
	S(H,W) = \bigcup_{{F\in\mathcal{H}}, \, {F\subsetneq H}}N(F,W).
	\end{equation*}
	Note that $N_G(W)N(H,W)=N(H,W)= N(H,W)N_G(H)$. By \cite[Proposition 2.1, Lemma 2.4]{MS95},
	\begin{equation*}
	N(H,W)\cap N(H,W)\gamma \subset S(H,W), \; \forall \gamma\in\Gamma\backslash N_G(H).
	\end{equation*}
	By Ratner's theorem \cite[Theorem 1]{Ra91}, as explained in \cite[Theorem 2.2]{MS95}, we have the following.
	\begin{thm}[Ratner] \label{thm:Ratner}
		Given a $W$-invariant probability measure $\lambda$ on $X$, there exists $H\in\mathcal{H}$ such that
		\begin{equation*}\label{eq:positive_limit_measure_on_singular_set}
		\lambda(\pi(N(H,W)))>0 \quad \text{and} \quad \lambda(\pi(S(H,W)))=0.
		\end{equation*}
		Moreover, almost every $W$-ergodic component of $\lambda$ restricted to $\pi(N(H,W))$ is a measure of the form $g\mu_H$, where $g\in N(H,W)\backslash S(H,W)$ and $\mu_H$ is a finite $H$-invariant measure on $\pi(H)\cong H/H\cap\Gamma$.
		
		Further, if $H$ as above is a normal subgroup of $G$, then $\lambda$ is $H$-invariant. 
		\end{thm}
		
		To justify the last sentence in \Cref{thm:Ratner}, note that $\lambda(\pi(N(H,W)))>0$, so $N(H,W)\neq\emptyset$. Since $N_G(H)=G$, we have $N(H,W)=N(H,W)N_G(H)=G$, and hence $\lambda$ restricted to $\pi(N(H,W))$ equals $\lambda$. And for every $g\in G$, $g\mu_H$ is $H$-invariant. So almost every $W$-ergodic component of $\lambda$ is $H$-invariant, so $\lambda$ is $H$-invariant. 
	 
	Now let $H\in \cH$ and put $d = \dim H$. Let $\mathfrak{g}$ denote the Lie algebra of $G$ and take $V=\bigwedge^d\mathfrak{g}$.
	Then $V$ admits a $\bbQ$-structure corresponding to the standard $\bbQ$-structure on $\mathfrak{g}$.
	Also $G$ acts on $V$ via the adjoint action of $G$ on $\mathfrak{g}$. 
	Since $H$ is defined over $\bbQ$, its Lie algebra $\mathfrak{h}$ is a $\bbQ$-subspace of $\mathfrak{g}$.
	Fix $p_H \in \bigwedge^d\mathfrak{h}(\bbQ)\backslash \{0\}$. 
	Then the orbit $\Gamma p_H$ is a discrete subset of $V$. 
	We note that for any $g\in N_G(H)$, $gp_H = \det(\mathrm{Ad}\,g\vert_\mathfrak{h})p_H$. 
	Hence the stabilizer of $p_H$ in $G$ equals
	\begin{equation*}
	N_G^1(H) := \{ g\in N_G(H)\colon \det(\mathrm{Ad}\,g\vert_\mathfrak{h})=1\}.
	\end{equation*}
	Fix $w_0\in\mathfrak{g}$ such that $\mathrm{Lie}(W)=\bbR w_0$, 
	and for $V$ as above define
	\begin{equation*}
	\mathcal{A} = \{ v\in V \colon v\wedge w_0 = 0 \}.
	\end{equation*}
	Then $\mathcal{A}$ is a linear subspace of $V$ and we observe that
	\begin{equation*}
	N(H,W) = \{ g\in G\colon g\cdot p_H \in \mathcal{A} \}.
	\end{equation*}
	
	By the linearization technique \cite[Proposition 4.2]{DM93} we obtain the following:
	
	\begin{prop}\label{prop:nonfocusing_criterion}
		Let $\cC$ be a compact subset of $N(H,W)\setminus S(H,W)$. Given $\eps > 0$, there exists a compact set $\mathcal{D}\subset\mathcal{A}$ such that, given a neighborhood $\Phi$ of $\mathcal{D}$ in $V$, there exists a neighborhood $\mathcal{O}$ of $\pi(\cC)$ in $X$ such that for any $t\in\bbR$ and any subinterval $J\subset I$, one of the following statements holds:
		\begin{enumerate}
			\item[\rm (1)] $\lvert \{ s\in J \colon g_t\phi_{a,b}(s)x_0 \in \mathcal{O} \} \rvert \leq \eps \lvert J \rvert$.
			\item[\rm (2)] There exists $\gamma\in\Gamma$ such that $g_t\phi_{a,b}(s)\gamma p_H \in \Phi$ for all $s\in J$.
		\end{enumerate}
	\end{prop}
	
	Let $\lambda_{a,b}$ be as in \eqref{eq:lambda_ab}.
	
	\begin{prop}\label{prop:equidistribution}
		Let $\mu$ be a weak-* limit of $g_{t_i}\lambda_{a,b}$ for a sequence $t_i\to\infty$. Suppose $\mu$ is not the $G$-invariant probability measure $\mu_{X}$. Then there exists $R>0$ and a sequence $\{v_i\}\subset\bbZ^3\setminus\{0\}$ such that
		\begin{equation} \label{eq:bounded}
		\sup_{s\in I}\norm{g_{t_i}\phi_{a,b}(s)v_i}\leq R.
		\end{equation}
		In particular $(a,b)\in\Wtwo$.
	\end{prop}
	
	\begin{proof}
	If $\mu$ is not a probability measure on $X$, then condition~(2), and hence condition~(3), of \Cref{prop:consequence_of_divergence} hold. So \eqref{eq:bounded} follows.
	
	Therefore, we now assume that $\mu$ is a probability measure on $X$. 
		By \Cref{prop:U-inv} $\mu$ is $W$-invariant. By \Cref{thm:Ratner}, there exists $H\in \cH$ such that \[ 
		\mu(\pi(N(H,W)))>0 \text{ and } \mu(\pi(S(H,W)))=0,
		\]
		and since $\mu$ is not $G$-invariant, $H\neq G$. So $\dim(H)<\dim(G)$. 
		
		We thus conclude that there exist $\eps>0$ and a compact set $\cC\subset N(H,W)\setminus S(H,W)$ such that $\mu(\pi(\cC))>\eps$.
		By \Cref{prop:nonfocusing_criterion} applied to $\eps/2$ in place of $\eps$, we obtain a compact set $\cD\subset\cA$. Then we pick $R_1>0$ such that $\cD$ is contained in the open norm-ball of radius $R_1$ in $V$, denoted $\Phi$, and obtain a neighborhood $\cO$ of $\pi(\cC)$ in $G/\Gamma$ so that the conclusion of \Cref{prop:nonfocusing_criterion} holds. 
		
		Since $\mu(\pi(\cC))>\eps$, there exists $i_0\in\bbN$ such that for all $i\geq i_0$, $g_{t_i}\lambda_{a,b}(\cO)>\eps$. So for $i\geq i_0$, $t=t_i$ and $J=I$ condition~(1) of the conclusion of \Cref{prop:nonfocusing_criterion} fails to hold, and hence condition~(2) of the conclusion holds for some $\gamma_i\in\Gamma$; that is,
		\begin{equation*}
		    \sup_{s\in I}\norm{g_{t_i}\phi_{a,b}(s)(\gamma_i p_H)}\leq R_1.
		\end{equation*}
		Therefore \eqref{eq:bounded} follows from \Cref{prop:consequence_of_linear_focusing} for a choice of $R>0$ depending on $p_H\in V(\bbQ)\setminus\{0\}$ and $R_1>0$. 
	\end{proof}

	\begin{prop}\label{prop:Ave-equidistribution}
		Let $\mu$ be a weak-* limit of $\mu_i:=(1/T_i)\int_{0}^{T_i}g_{t_i}\lambda_{a,b}\,\mathrm{d}t$ for a sequence $T_i\to\infty$; here $0\leq \mu(X)\leq 1$. Suppose $\mu \ne \mu_X$. 
		Then there exists $R>0$ such that 
		\begin{equation} \label{eq:Ave-bounded}
		   \liminf_{I\to\infty} \frac{\abs{\cI_R\cap[0,T_i]}}{T_i}>0,
		\end{equation}
		where $\cI_{R}$ is defined in \eqref{eq:defOfIR}.
	\end{prop}
	
	\begin{proof} There are two possibilities: $\mu(X)<1$, or $\mu$ is a probability measure which is not $G$-invariant. By \Cref{prop:Ave-W-inv}, $\mu$ is $W$-invariant. So there exists $\eps>0$ such that one of the following two possibilities occur:
	\begin{enumerate}
\item[(i)] $\mu(X)<1-\eps$;
\item[(ii)] or, by \Cref{thm:Ratner}, there exists $H\in\cH$  with 
	$H\neq G$ and a compact set $\cC\subset N(H,W)\setminus S(H,W)$ such that $\mu(\pi(\cC))>\eps$. 
	\end{enumerate}
	
	First suppose that possibility (i) occurs. Take any $R>0$ and pick a compact $K\subset X$ given by \Cref{thm:KMnondivergence} for $\eps/2$ in place of $\eps$.
	Then, for each non-negative $t\not\in \cI_R$, by definition~\eqref{eq:defOfIR}, the possibility (2) of \Cref{thm:KMnondivergence} does not hold, and hence its possibility (1) must hold; that is, $g_t\lambda_{a,b}(K)\geq 1-\eps/2$. Write $\kappa_i=\abs{\cI_R\cap[0,T_i]}/{T_i}$. So for all large $i$,
	\[
	(1-\kappa_i)(1-\eps/2)\leq \frac{1}{T_i}\int_{0}^{T_i} g_t\lambda_{a,b}(K)\,\mathrm{d}t \leq \mu_i(X)< 1-\eps;
	\]
	 and hence $\kappa_i>\eps/2$. So \eqref{eq:Ave-bounded} holds.
	
	Now suppose possibility (ii) occurs. Then for any open neighborhood $\cO$ of $\pi(\cC)$, $\mu(\cO)>\eps$, and so for all large $i$, 
	\[
	\frac{1}{T_i}\int_{0}^{T_i} g_t\lambda_{a,b}(\cO)\,\mathrm{d}t > \eps.
	\]
	Let 
	\[
	\cI_\cO=\{t\in[0,\infty): g_t\lambda_{a,b}(\cO)>\eps/2\} \text{ and }
	\kappa_i=\frac{\abs{\cI_{\cO}\cap [0,T_i]}}{T_i}.
	\]
	Then for all large $i$, 
	\[
	(1-\kappa_i)\eps/2+\kappa_i\geq \frac{1}{T_i}\int_{0}^{T_i} g_t\lambda_{a,b}(\cO)\,\mathrm{d}t>\eps,
	\]
	and hence $\kappa_i> \eps/2$.
	
	By \Cref{prop:nonfocusing_criterion} applied to the set $\cC\subset N(H,W)\setminus S(H,W)$ and $\eps/2$ in place of $\eps$, we obtain a compact set $\cD\subset \cA$. Pick $R_1>0$ such that $\cD$ is contained in the open norm-ball of radius $R_1$ in $V$, denoted $\Phi$, and obtain a neighborhood $\cO$ of $\pi(\cC)$ in $G/\Gamma$ so that the conclusion of \Cref{prop:nonfocusing_criterion} holds. 
	
	Now suppose $t\in \cI_{\cO}$. Then for $J=I$, the condition (1) of \Cref{prop:nonfocusing_criterion} fails to hold, so condition (2) of  \Cref{prop:nonfocusing_criterion} holds: there exists $\gamma\in\Gamma$ such that  
	\begin{equation} \label{eq:bndC}
	    \sup_{s\in I} \norm{g_t\phi_{a,b}(s)\gamma p_H}<R_1.
	\end{equation}
	Let $R>0$ be the quantity given by \Cref{prop:consequence_of_linear_focusing} applied to $v_0=p_H\in V(\bbQ)\setminus\{0\}$ and $C=R_1$.
	Under \eqref{eq:bndC}, \Cref{prop:consequence_of_linear_focusing} shows that there exists $v\in\bbZ^3\setminus\{0\}$ such that $\sup_{s\in I} \norm{g_t\phi_{a,b}(s)v}\leq R$; that is, $t\in I_R$.
	Thus $I_\cO\subset I_R$. Therefore 
	\[
	\abs{I_R\cap[0,T_i]}/T_i\geq \kappa_i> \eps/2
	\]
	for all large $i$, and \eqref{eq:Ave-bounded} follows.
	\end{proof}
	
	\section{Proof of {Theorem~\ref{thm:main_thm_equidistribution} and Theorem~\ref{thm:main_thm_sequence}}}
	
	In this section we prove \Cref{thm:main_thm_equidistribution} and \Cref{thm:main_thm_sequence}.
	
	\begin{proof}[\bf Proof of \Cref{thm:main_thm_equidistribution}]
First suppose that the $g_t$-translates of $\lambda_{a,b}$ do \emph{not} get equidistributed in $X$ as $t\to\infty$. Then there exist $t_i\to\infty$ such that $g_{t_i}\lambda_{a,b}$ weak-* converge to a measure which is not $\mu_{X}$. So by \Cref{prop:equidistribution} we have $(a,b)\in\Wtwo$. 
		
		Conversely, suppose $(a,b)\in \Wtwo$. We want to show that the $g_t$-translates of $\lambda_{a,b}$ do \emph{not} get equidistributed in $X$.
		
		By \Cref{lem:interpretation_of_Wtwo}, there exist $R\geq 1$, $t_i\to\infty$ and $\{\gamma_i\}\subset \Gamma$ such that for all $i\in\bbN$,
		\begin{equation} \label{eq:temp11} 
		\sup_{s\in I}\norm{g_{t_i}\phi_{a,b}(s)\gamma_ie_1}\leq R.
		\end{equation}
		
		\subsubsection*{Case 1:}  Suppose there exists $c>0$ such that for all $i\geq1$ and all $s\in I$,
		\[
		\norm{g_{t_i}\phi_{a,b}(s)\gamma_ie_1}\geq c.
		\]
		Then by \Cref{prop:consequence_of_divergence}, after passing to a subsequence, we may assume that $g_{t_i}\lambda_{a,b}$ weak-* converge to a probability measure $\mu$. It suffices to show that the support of $\mu$ is not full.
		
		Let $E$ denote the set of unimodular lattices in $\bbR^3$ containing a primitive vector whose (sup)norm is in the interval $[c,R]$. Then $E$ is closed and contains the support of each $g_{t_i}\lambda_{a,b}$. Therefore the support of $\mu$ is also contained in $E$. But $X\setminus E$ is a nonempty open set, as $E$ does not contain the unimodular lattice $\bbZ M^{-2}e_1 + \bbZ Me_2 + \bbZ Me_3$, for any $M>R$ such that $M^{-2}<c$. Thus the support of $\mu$ is not full. 
		\subsubsection*{Case 2:} Suppose Case~1 does not occur. Then
		after passing to a subsequence, there exists a sequence $\{s_i\}\subset I$ such that
		\begin{equation} \label{eq:temp12b}
		\norm{g_{t_i}\phi_{a,b}(s_i)\gamma_ie_1}=c_i\to0 \text{ as $i\to\infty$}.
		\end{equation}
		We write $\gamma_ie_1=\begin{pmatrix}p_{1,i}\\ p_{2,i}\\ q_i\end{pmatrix}\in\bbZ^3$. Then for all $s\in I$ and $t\in\R$
		\begin{equation} \label{eq:coordinates1}
		g_t\phi_{a,b}(s)\gamma_ie_1=\begin{pmatrix}e^{2t}x_1(s)\\ e^{-t}p_{2,i}\\ e^{-t}q_i
		\end{pmatrix},
		\end{equation}
		where $x_1(s)=(aq_i+p_{2,i})s+(bq_i+p_{1,i})$.
		So by \eqref{eq:temp12b}, $e^{-t_i}\abs{p_{2,i}}\leq c_i$ and $e^{-t_i}\abs{q_i}\leq c_i$, and by \eqref{eq:temp11}, $\abs{e^{2t_i}x_1(s)}\leq R$. 
		We note that if $p_{2,i}=0$ and $q_i=0$, then $\abs{x_1(s)}=\abs{p_{1,i}}\geq 1$, and hence $e^{2t_i}\leq R$ for all large $i$, which is absurd. Hence $$c_ie^{t_i}\geq \max\{\abs{p_{2,i}},\abs{q_i}\}\geq 1;$$ that is, $t_i+\log c_i\geq 0$ for all $i$. 
		Let $t_i':=t_i+(1/3)\log c_i\geq -(2/3)\log c_i$, so $t_i'\to\infty$.
		By \eqref{eq:coordinates1}, for all $s\in I$,
		\begin{align*} 
		\norm{g_{t_i'}\phi_{a,b}(s)\gamma_ie_1}&\leq\max\left\{
		\abs{c_i^{2/3}e^{2t_i}x_1(s)},\abs{c_i^{-1/3}e^{-t_i}p_{2,i}},\abs{c_i^{-1/3}e^{-t_i}q_i}\right\}\\
		&\leq \max\left\{Rc_i^{2/3},c_i^{-1/3}c_i\right\}\leq Rc_i^{2/3}.
		\end{align*}
		Since $c_i\to 0$, by \Cref{prop:consequence_of_divergence},
		for any compact $K\subset X$ we have $g_{t_i'}\lambda_{a,b}(K)=0$ for all large $i$.
		\end{proof}
	
	\begin{proof}[\bf Proof of \Cref{thm:main_thm_sequence}]
(1)$\Rightarrow$(2): Assume $(a,b)\notin\bbQ^2$. Then (3) of \Cref{lem:inter_of_Q}  fails to hold, so there exists a sequence $t_i\to\infty$ such that for every sequence $\{v_i\}\subset \bbZ^3\setminus\{0\}$,
		\begin{equation*} 
		\sup_{s\in I} \norm{g_{t_i}\phi_{a,b}(s)v_i}\to\infty \text{ as $i\to\infty$}.
		\end{equation*}
		So by \Cref{prop:equidistribution}, we conclude that $g_{t_i}\lambda_{a,b}\stackrel{\text{weak-*}}{\longrightarrow}\mu_{X}$.
 
 (2)$\Rightarrow$(1): If $(a,b)\in \bbQ^2$, then by Remark~\ref{rem:Q-diverge}, the translated measure $g_t\lambda_{a,b}$ escapes to $\infty$ as $t\to\infty$. Therefore (2) fails to hold. 
 
Thus (1) and (2) are equivalent. Next we will prove that (1)$\Rightarrow$(3) and (3)$\Rightarrow$(1).

		(1)$\Rightarrow$(3):  
		We assume (1) and argue by contradiction, supposing that the set
		\[
		E=\left\{ s\in \bbR : \{ g_t\phi_{a,b}(s)x_0 \}_{t\geq 0} \text{ is not dense in } X \right\}
		\]
		has positive Lebesgue measure. We take a countable topological basis $\{B_i\}_{i\in\bbN}$ of $X$ consisting of non-empty open subsets, and let
		\[
		E_i=\left\{ s\in\bbR : \{ g_t\phi_{a,b}(s)x_0 \}_{t\geq 0}\cap B_i = \emptyset \right\}.
		\]
		One has $E=\bigcup_{i\in\bbN}E_i$. Since $\abs{E}>0$, there exists $i_0\in\bbN$ such that $\abs{E_{i_0}}>0$. Without loss of generality we may assume that $\abs{E_1}>0$. By the Lebesgue density theorem, there exists a compact interval $I\subset \bbR$ with non-empty interior such that
		\begin{equation}\label{eq:LebesgueDensity}
		\frac{\abs{I\cap E_1}}{\abs{I}}\geq 1-\frac{\mu_{X}(B_1)}{2},
		\end{equation}
		as $\mu_X(B_1)>0$. Because we assumed (1), and since we have proved that (1)$\Rightarrow$(2),  there exists $t_i\to\infty$ such that $g_{t_i}\lambda_{a,b}\to\mu_X$ in the weak-* topology. Since $B_1$ is non-empty and open, for all large $i$,
		\[
		\frac{1}{\abs{I}}\abs{\{s\in I : g_{t_i}\phi_{a,b}(s)x_0\in B_1 \}} > \frac{\mu_{X}(B_1)}{2},
		\]
		which, by the definition of $E_1$, implies that
		\[
		\frac{\abs{I\setminus E_1}}{\abs{I}}> \frac{\mu_{X}(B_1)}{2}.
		\]
		This contradicts \eqref{eq:LebesgueDensity}. Hence we must have $\abs{E}=0$.
		
		(3)$\Rightarrow$(1): To prove this by contraposition, suppose that $(a,b)\in\bbQ^2$. Let $B\subset X$ be a non-empty relatively compact open set. By Remark~\ref{rem:Q-diverge}, there exists $t_0>0$ such that $g_t\{\phi_{a,b}(s)x_0:s\in \R\}\cap B=\emptyset$ for all $t>t_0$. If $q\in\bbN$ be such that $qa\in\bbZ$, then $\phi_{a,b}(s+q)\bbZ^3=\phi_{a,b}(s)\bbZ^3$ for all $s\in\R$. Therefore $\{\phi_{a,b}(s)x_0:s\in\R\}$ is compact. So $C:=\cup_{0\leq t\leq t_0}g_t\{\phi_{a,b}(s)x_0:s\in\R\}$ is a compact subset of a $2$-dimensional submanifold of $X$. So $B\setminus C$ is a non-empty open subset of $X$. Therefore for every $s\in\bbR$, 
		\[
		\{g_t\phi_{a,b}(s)x_0:t\geq 0\}\cap (B\setminus C)=\emptyset;
		\]
		in particular, $\{g_t\phi_{a,b}(s)x_0:t\geq 0\}$ is not dense in $X$. So (3) fails to hold.
	\end{proof}
	
	\section{Behavior on average --- Proofs of {Lemma~\ref{lem:density}} and  {Theorem~\ref{thm:main_thm_average}}}
	\label{sec:averages}
	In this section, we discuss the averages of the $g_t$-translates, and prove \Cref{thm:main_thm_average}. As \Cref{lem:density} will be used in the proof of \Cref{thm:main_thm_average}, we first provide its proof. 
	
	\begin{proof}[\bf Proof of \Cref{lem:density}]
	(1)$\Rightarrow$(2) is obvious.
		
		To prove that (2)$\Rightarrow$(3), we pick $R\geq 1$ such that \eqref{eq:positive_upper_density} holds. 
		For any 
		 $v=\begin{pmatrix}p_1\\ p_2\\ q\end{pmatrix}\in \bbZ^3\setminus\{0\}$ and $t\geq 0$, by \eqref{eq:temp6b} we get
		\begin{equation} \label{eq:temp6b3}
		\sup_{s\in I}\norm{g_t\phi_{a,b}(s)v} < R \Rightarrow
		\begin{cases}
		\abs{qb+p_{1}}< R_1e^{-2t} \\
		\abs{qa+p_{2}}<R_1e^{-2t} \\
		\abs{q}< Re^{t}
		\end{cases}\Rightarrow 
		\begin{cases}
		\lip q\begin{pmatrix}
		b\\a \end{pmatrix}\rip<R_1e^{-2t}\\
		\abs{q}< Re^{t},
		\end{cases}
		\end{equation}
		where $R_1=\norm{
		\begin{pmatrix}
		1 & s_0 \\ 1 & s_1
		\end{pmatrix}^{-1}}R\geq R$, and 
		$\lip \begin{pmatrix}
	x_1\\x_2
	\end{pmatrix} \rip$ denotes the sup-norm distance between $\begin{pmatrix}
	x_1\\x_2
	\end{pmatrix}$ and its nearest integral vector. Note that if $q=0$ in \eqref{eq:temp6b3}, then $t<(1/2)\log R_1$. 
	
	For each $q\in\bbN$, define
		\begin{align}
		E_q&=\{t>0: e^{-t}q<R \text{ and } e^{2t}\lip q\begin{pmatrix}
		b\\a \end{pmatrix}\rip<R_1 \} \label{eq:definitionOfEq}\\
		&=\bigl\{t>0: R^{-1}q<e^t<R_1^{1/2}\lip q\begin{pmatrix}
		b\\a \end{pmatrix}\rip^{-1/2}\bigr\} \nonumber\\
		&=\bigl(\log q-\log R, -\frac{1}{2}\log \lip q\begin{pmatrix}
		b\\a
		\end{pmatrix}\rip + \frac{1}{2}\log R_1\bigr)\cap (0,\infty). \nonumber
		\end{align}
		Now $t\in E_q$ if and only if the right-most term of \eqref{eq:temp6b3} holds, so
		\begin{equation} \label{eq:Eqnonempty}
		    E_q\neq\emptyset \iff
		\lip q\begin{pmatrix}
		b\\a
		\end{pmatrix}\rip < R_1(Rq^{-1})^2.
		\end{equation}
		
		Let $\cP(\bbZ^3)$ denote the set of primitive integral vectors in $\bbZ^{3}$.  From \eqref{eq:defOfIR}, note that
		\begin{equation} \label{eq:definitionOfIR}
		\cI_{R}=\left\{ t\in[0,+\infty) : \sup_{s\in I}\norm{g_t\phi_{a,b}(s)v} < R, \text{for some } v\in\cP(\bbZ^3)\setminus\{0\} \right\}.
		\end{equation}
		Let $\cQ$ be the collection of 
		$q\in\bbN$ such that $E_q\neq\emptyset$, and 
		\[
		\lip q\begin{pmatrix}
		b\\a \end{pmatrix}\rip = \max\{\abs{qb+p_1},\abs{qa+p_2}\} \text{ for some } p_1,p_2\in\bbZ \text{ such that } \begin{pmatrix}p_1\\p_2\\q\end{pmatrix}\in \cP(\bbZ^3).  
		\]
		Let $t_0=(1/2)\log R_1$. Then by \eqref{eq:definitionOfIR}, 
		\eqref{eq:temp6b3} and \eqref{eq:definitionOfEq} we get
		\begin{equation} \label{eq:inUnionOfPrimitiveEq}
		\cI_R\cap [t_0,\infty]\subset\bigcup_{q\in \cQ}E_q.
		\end{equation}
		
		We may assume that $(a,b)\not\in\bbQ^2$, because otherwise $(a,b)\in \Wtwoo$ and we are done. Then for any $q\in\cQ$, $E_q$ is a finite interval. Therefore by \eqref{eq:positive_upper_density}, $\cQ$ is infinite. We write $\cQ=\{q_1,q_2,\dots\}$, where $q_i<q_{i+1}$ for all $i$. 
		
		\subsubsection*{Claim 1} {\it If $E_q$ and $E_{q'}$ are both non-empty for some $q,q'\in\cQ$ and $q'>q$, then $q^2<Cq'$, where $C=2R_1R^2\geq 2$. In particular, for any $n\in\bbN$,
		\begin{equation} \label{eq:growth}
		    \log q_{i}<\log C+2^{-(n-i)}\log q_n,\,\forall i<n.
		\end{equation}
		}
		
		Indeed, by definition of $\cQ$ and \eqref{eq:Eqnonempty}, there exist $\begin{pmatrix}
		p_1\\p_2\\q
		\end{pmatrix}, \begin{pmatrix}
	p_1'\\p_2'\\q'
		\end{pmatrix}\in \cP(\bbZ^3)$ such that 
		\begin{equation*}
		\norm{q\left(\begin{matrix}
			b\\a
			\end{matrix}\right)+\left(\begin{matrix}
			p_1\\p_2
			\end{matrix}\right)}<R_1R^2q^{-2} \text{ and }
		\norm{q'\left(\begin{matrix}
			b\\a
			\end{matrix}\right)+\left(\begin{matrix}
			p'_1\\p'_2
			\end{matrix}\right)}<R_1R^2 q'^{-2}.
		\end{equation*}
		 By primitivity, $\frac{1}{q}\begin{pmatrix} p_1\\p_2\end{pmatrix}\neq \frac{1}{q'}\begin{pmatrix}
		 p_1'\\p_2'
		 \end{pmatrix}$. Hence, by triangular inequality,
		\begin{equation*}
		\frac{1}{qq'}\leq\norm{\frac{1}{q}\left(\begin{matrix}
		p_1\\p_2
		\end{matrix}\right)-\frac{1}{q'}\left(\begin{matrix}
		p_1'\\p'_2
		\end{matrix}\right)}\leq R_1R^2(q^{-3}+q'^{-3})< (2R_1R^2)q^{-3}. 
		\end{equation*}
		Therefore $q^2<(2R_1R^2)q'$. This proves the first part of the claim. 
		
		For the second assertion of the claim, we iteratively apply the inequality $q^2<Cq'$ to $q=q_j$ and $q'=q_{j+1}$, for $j=i,\dots,n-1$ to get
		\begin{align*}
		\log q_i & < \frac{\log C}{2} + \frac{\log q_{i+1}}{2}\\
		& < \frac{\log C}{2} +\frac{\log C}{4} + \frac{\log q_{i+2}}{4}\\
		& < \dots \\
		& < \log C + 2^{-(n-i)}\log q_n.
		\end{align*}
		
		Next, in view of \eqref{eq:positive_upper_density} and \eqref{eq:inUnionOfPrimitiveEq}, to achieve the quantity $\limsup_{T\to\infty}\frac{\abs{[0,T]\cap \bigcup_{q\in\cQ}E_{q}}}{T}$, it is enough to let $T$ vary along the sequence $\{T_n\}$ of right endpoints of intervals $E_{q_n}$, which, by \eqref{eq:definitionOfEq}, can be rewritten $T_n=\log q_n-\log R+\abs{E_{q_n}}$.
		Then,
		\[
		\frac{\abs{[0,T_n]\cap \bigcup_{q\in\cQ}E_{q}}}{T_n} \leq \frac{\sum_{i=1}^n\abs{E_{q_i}}}{\log q_n-\log R+\abs{E_{q_n}}}.
		\]
		Therefore we infer from \eqref{eq:positive_upper_density} and \eqref{eq:inUnionOfPrimitiveEq} that
		\begin{equation*}
		\limsup_{n\to\infty}\frac{\sum_{i=1}^n\abs{E_{q_i}}}{\log q_n-\log R+\abs{E_{q_n}}} > 0.
		\end{equation*}
		It follows that there exists $\eps>0$ such that
		\begin{equation}\label{eq:positiveUpperDensity}
		\limsup_{n\to\infty}\frac{\sum_{i=1}^n\abs{E_{q_i}}}{\log q_n} = 4\eps > 0.
		\end{equation}
		
		\subsubsection*{Claim 2}{\it We claim that
		\begin{equation} \label{eq:positiveUpperDensityEn}
		\limsup_{n\to\infty}\frac{\abs{E_{q_n}}}{\log q_n} > \eps.
		\end{equation}}
		
		Indeed, suppose $\limsup_{n\to\infty}\frac{\abs{E_{q_n}}}{\log q_n} \leq \eps$. Then there exists $N>0$ such that $q_N>C$ and for all $n\geq N$, $\abs{E_{q_n}}<2\eps\log q_n$.  Therefore
		\begin{equation*}
		\begin{split}
		&\limsup_{n\to\infty}\frac{\sum_{i=1}^n\abs{E_{q_i}}}{\log q_n}=\limsup_{n\to\infty}\frac{\sum_{i=N}^n\abs{E_{q_i}}}{\log q_n}
		<\limsup_{n\to\infty}\frac{\sum_{i=N}^n2\eps\log q_i}{\log q_n}\\
		{\leq}&\limsup_{n\to\infty}\frac{2\eps(n-N)\log C+2\eps\sum_{i=N}^n2^{-(n-i)}\log q_n}{\log q_n}= 0+2\epsilon\sum_{i=N}^n 2^{-(n-i)}<4\epsilon,
		\end{split}
		\end{equation*}
		because by \eqref{eq:growth}, for any $i< n$,
		\[
		\log q_i< \log C+2^{-(n-i)}\log q_n \text{ and } \log q_n>2^{(n-N)}(\log q_N-\log C).
		\]
		This contradicts \eqref{eq:positiveUpperDensity}, and proves Claim~2.
		
		\bigskip
		Now in view of \eqref{eq:positiveUpperDensityEn}, for any $Q>0$ there exists   $q>Q$ such that $\abs{E_q}>\eps\log q$. By \eqref{eq:definitionOfEq}, this means
		\begin{equation*}
		(1/2)\log R_1+\log R - \log q - \frac{1}{2}\log \lip q\begin{pmatrix}
		b\\a
		\end{pmatrix}\rip >\eps\log q,
		\end{equation*}
		or equivalently,
		\[
		\lip q\begin{pmatrix}
		b\\a
		\end{pmatrix}\rip < R_1R^2 q^{-(2+2\eps)}.
		\]
		Hence $\lip q\begin{pmatrix}
		b\\a
		\end{pmatrix}\rip\leq q^{-(2+\eps)}$ has infinitely many solutions $q\in\bbN$, which means  $(a,b)\in\Wtwoo$. This proves (2)$\Rightarrow$(3). 
		
		\bigskip
		Now to prove (3)$\Rightarrow$(1), suppose that $(a,b)\in\Wtwoo$. Then there exists $\eps>0$ and an increasing sequence $\{q_n\}_{n\in\bbN}$ of positive integers such that
		\begin{equation}\label{eq:Wtwoo2}
		\begin{cases}
		\abs{q_nb+p_{1,n}}\leq q_n^{-(2+\eps)} \\
		\abs{q_na+p_{2,n}}\leq q_n^{-(2+\eps)}
		\end{cases},\text{ for some }p_{1,n},p_{2,n}\in\bbZ.
		\end{equation}
		For each $n\in\bbN$, pick $v_n=\begin{pmatrix}
		p_{1,n}\\p_{2,n}\\q_n
		\end{pmatrix}\in\bbZ^3$ such that \eqref{eq:Wtwoo2} holds. For any $t\in\bbR$,
		\begin{equation} \label{eq:temp6a2}
		g_{t}\phi_{a,b}(s)v_n=\begin{pmatrix}
		e^{2t}\bigl((bq_n+p_{1,n})+(aq_n+p_{2,n})s\bigr)\\e^{-t}p_{2,n}\\e^{-t}q_n
		\end{pmatrix}.
		\end{equation}
		Pick any constants $0<c_1<c_2<1/2$, independent of $n$. Let 
		\[
		t\in[(1+c_1\eps)\log q_n,(1+c_2\eps)\log q_n].
		\]
		Then
		\begin{equation*}
		q_n^{-(2+\eps)}\leq e^{-\bigl(\frac{(1-2c_2)\eps}{1+c_2\eps}\bigr) t}e^{-2t} \text{ and } q_n\leq e^{-\bigl(\frac{c_1\eps}{1+c_1\eps}\bigr) t}e^{t}.
		\end{equation*}
		By \eqref{eq:Wtwoo2},
		\[
		\abs{p_{2,n}}\leq q_n^{-(2+\eps)}+\abs{a}q_n\leq 1+\abs{a}q_n
		\]
		and so, by \eqref{eq:temp6a2}, 
		\begin{equation}\label{eq:temp10}
		\norm{g_t\phi_{a,b}(s)v_n}\leq  C_1e^{-\eps_1 t}, \,\forall s\in I=[s_0,s_1],
		\end{equation}
		where $\eps_1:=\min\left\{\frac{c_1\eps}{1+c_1\eps},\frac{(1-2c_2)\eps}{1+c_2\eps}\right\}>0$ and $C_1:=(1+\abs{s_0}+\abs{s_1}+\abs{a})$.
		
		Given any $R>0$, let $N>0$ such that for every $n>N$, 
		\[
		C_1q_n^{-(1+c_1\eps)\eps_1}<R.
		\]		
		For $n>N$, by \eqref{eq:temp10}, one has $[(1+c_1\eps)\log q_n,(1+c_2\eps)\log q_n]\subset \cI_{R}$.
		So, setting $T_n=(1+c_2\eps)\log q_n$, we get 
		\[
		\frac{\abs{\cI_{R}\cap [0,T_{n}]}}{T_{n}}\geq \frac{\abs{[(1+c_1\eps)\log q_n,(1+c_2\eps)\log q_n]}}{T_{n}}=\frac{(c_{2}-c_{1})\eps}{1+c_2\eps}.
		\]
		Therefore
		\[
		\limsup_{T\to\infty}\frac{\abs{\cI_R\cap[0,T]}}{T}\geq \frac{(c_{2}-c_{1})\eps}{1+c_2\eps}>0.
		\]
		This proves that (3)$\Rightarrow$(1). 
	\end{proof}
	
	Now we are ready to prove \Cref{thm:main_thm_average}.
	\begin{proof}[\bf Proof of \Cref{thm:main_thm_average}]
		 (1)$\Rightarrow$(2) is obvious.
		
		To prove (2)$\Rightarrow$(3) by contrapositive, suppose that $(a,b)\in\Wtwoo$. Let $K$ be a compact subset of $X$, which, as we may recall,  is identified with the space of unimodular lattices in $\bbR^{3}$. By Mahler's criterion, there exists $R>0$ such that every nonzero vector in any lattice in $K$ has norm at least $R$.
		So, by \eqref{eq:defOfIR},  for any $t\in \cI_{R}$, we have $g_{t}\phi_{{a,b}}(s)\bbZ^{n}\notin K$ for all $s\in I$; in particular, 
		\begin{equation} \label{eq:gtlambdaK}
		g_{t}\lambda_{{a,b}}(K)=0.
		\end{equation} 
		
		Since $(a,b)\in\Wtwoo$, \Cref{lem:density} shows that there exists a sequence $T_n\to\infty$ and an $\eps>0$ such that for all $n$,
		\[
		\frac{\abs{\cI_{R}\cap [0,T_{n}]}}{T_{n}}\geq \epsilon,
		\]
		and hence, by \eqref{eq:gtlambdaK},
		\[
		\frac{1}{T_{n}}\int_{0}^{T_{n}}g_{t}\lambda_{{a,b}}(K)\,\mathrm{d}t\leq 1-\epsilon, 
		\]
		where the $\epsilon$ is independent of $K$. Thus, the family of averages $\{\frac{1}{T}\int_{0}^{T}g_{t}\lambda_{{a,b}}\,\mathrm{d}t\}_{T>0}$ has escape of mass. This proves that (2)$\Rightarrow$(3).
		
		To prove (3)$\Rightarrow$(1) by contraposition, suppose that (1) fails to hold. Then there exists a sequence $T_i\to\infty$ such that $\mu_i:=(1/T_i)\int_{0}^{T_i}g_t\lambda_{a,b}\,\mathrm{d}t$ does not converge to 
		$\mu_X$. Since the $\mu_i$ are probability measures, by passing to a subsequence, without loss of generality we may assume that $\mu_i$ converges to a Borel measure $\mu$ on $X$ which is not 
		$\mu_{X}$; here $0\leq \mu(X)\leq 1$. Then by \Cref{prop:Ave-equidistribution} there exists $R>0$ such that 
		\[
		\liminf_{i\to\infty} \frac{\abs{\cI_R\cap [0,T_i]}}{T_i}>0.
		\]
		Then by \Cref{lem:density}, we get $(a,b)\in \Wtwoo$, which contradicts~(3). 
		\end{proof}

\end{document}